\documentclass[11pt]{amsart}
\usepackage[dvips]{epsfig}
\usepackage{graphics}
\usepackage{latexsym}
\usepackage{verbatim}
\usepackage{amsmath}
\usepackage{amsthm}
\usepackage{amssymb}
\usepackage{hyperref}
\usepackage{euscript}
\usepackage{mathrsfs}
\usepackage{dsfont}

\def\B{\hat B}
\def\H{\cal H}

\def\Sinc{\mathop{\rm Sinc}\nolimits}

\newcommand{\ens}[1]{\mathbb{#1}}

\newcommand{\NN}{\mathds{N}}
\newcommand{\ZZ}{\mathds{Z}}
\newcommand{\RR}{\mathds{R}}

\newcommand{\PP}{\mathds{P}}
\newcommand{\F}{\mathcal{F}}
\newcommand{\N}{\mathcal{N}}
\newcommand{\T}{\mathcal{T}}

\newcommand{\Pf}{\mathcal{P}}
\newcommand{\Q}{\mathcal{Q}}

\def\cal{\mathcal}
\def\F{{\cal F}} %Fourier transform
\def\B{{\cal B}} %Fourier transform
\def\M{{\cal M}} %Diffusive transform
\def\R{{\cal R}} %Specular transform

\def\derpar#1#2{\frac{\partial#1}{\partial#2}}
\def\var{\varepsilon}

\newtheorem{theorem}{Theorem}[section]

\newtheorem{proposition}[theorem]{Proposition}

\newtheorem{remark}[theorem]{Remark}

\def\signff{\bigskip\bigskip\hspace{80mm}
\vbox{{\sc Francis Filbet\par\vspace{3mm}
UniversitÃ© de Lyon,\par
UL1, INSAL, ECL, CNRS \par 
UMR5208, Institut Camille Jordan,\par
43 boulevard 11 novembre 1918,\par
F-69622 Villeurbanne cedex,  FRANCE
\par\vspace{3mm}e-mail:} filbet@math.univ-lyon1.fr }}
\begin{document}

\title[On deterministic approximation of the Boltzmann equation in a bounded domain]
{On deterministic approximation of the Boltzmann equation  in a bounded domain}\thanks{The author is partially supported by the European Research Council ERC Starting Grant 2009,  project 239983-\textit{NuSiKiMo}}

\author{Francis Filbet}

\hyphenation{bounda-ry rea-so-na-ble be-ha-vior pro-per-ties
cha-rac-te-ris-tic}

\begin{abstract}
In this paper we present a fully deterministic method for the numerical solution to the Boltzmann equation of rarefied gas dynamics in a bounded domain for multi-scale problems. Periodic, specular reflection and diffusive boundary conditions are discussed and investigated numerically.  The collision operator is treated by a Fourier approximation of the collision integral, which guarantees spectral accuracy in velocity with a computational cost of $N\,\log(N)$, where $N$ is the number of degree of freedom in velocity space. This algorithm is coupled with a second order finite volume scheme in space and a  time discretization allowing to deal for rarefied regimes as well as their hydrodynamic limit. Finally, several numerical tests illustrate the efficiency and accuracy of the method for unsteady flows (Poiseuille flows, ghost effects, trend to equilibrium).
\end{abstract}

\maketitle


\medskip


\tableofcontents

\section{Introduction}
\label{sec1}
The Boltzmann equation describes the behavior of a dilute gas of
particles when the only interactions taken into account are binary
elastic collisions. It reads for $x \in \Omega\subset \RR^d$, $v \in \RR^d$  ($d \ge 2$):
 \begin{equation}
 \label{eq:Q}
 \derpar{f}{t} + v \cdot \nabla_x f = \frac{1}{\varepsilon}\Q(f),
 \end{equation}
where $f:=f(t,x,v)$ is the time-dependent particles distribution
function in the phase space. The parameter $\varepsilon>0$ is the 
dimensionless Knudsen number defined as the  ratio of the mean free path over
a typical length scale such as the size of the spatial domain, which
measures if the gas is rarefied. The Boltzmann collision operator $\Q$
is a quadratic operator local in $(t,x)$. The time $t$ and position $x$
act only as parameters in $\Q$ and therefore will be omitted in
its description
 \begin{equation} \label{eq:Q2}
 \Q (f)(v) = \int_{v_\star \in \RR^d}
 \int_{\sigma \in \ens{S}^{d-1}}  B(|v-v_\star|, \cos \theta) \,
 \left( f^\prime_\star f^\prime - f_\star f \right) \, d\sigma \, dv_\star.
 \end{equation}
%In~\eqref{eq:Q} we dropped the $x$ variable and 
We used the shorthand $f = f(v)$, $f_\star = f(v_\star)$,
$f^\prime = f(v^\prime)$, $f_\star^\prime = f(v_\star^\prime)$. The
velocities of the colliding pairs $(v,v_\star)$ and
$(v^\prime,v^\prime_\star)$ are related by
 \begin{equation*}
%\label{eq:rel:vit}
\left\{
\begin{array}{l}
\displaystyle{ v^\prime   = v - \frac{1}{2} \big((v-v_\star) 
- |v-v_\star|\,\sigma\big),} 
\vspace{0.2cm} \\
\displaystyle{ v^\prime_\star = v - \frac{1}{2} \big((v-v_\star) + |v-v_\star|\,\sigma\big),}
\end{array}\right.
 \end{equation*}
 with $\sigma\in \ens{S}^{d-1}$.  The collision kernel $B$ is a
 non-negative function which by physical arguments of invariance only
 depends on $|v-v_\star|$ and $\cos \theta = {u} \cdot \sigma$, where
 $u = (v-v_\star)$ and $\hat u = u /\vert u \vert$ is the normalized relative
 velocity. 
In this work we are concerned with {\em short-range interaction} models and we assume that $B$ is locally integrable. These assumptions are satisfied for the so-called {\em hard spheres  model} $B(u,\cos\theta)=|u|$, and it is known as {\em Grad's angular cutoff
  assumption} when it is (artificially) extended to interactions deriving from a power-law potentials. As an important benchmark model for the numerical simulation we therefore consider in this paper the so-called {\em variable hard spheres model} (VHS), which writes
\begin{equation}
\label{VHSkernel}
B(u, \cos \theta) = C_\gamma \, |u|^\gamma,
\end{equation}
for some $\gamma \in (0,1]$ and a constant $C_\gamma >0$.

Boltzmann's collision operator has the fundamental properties of
conserving mass, momentum and energy
 \begin{equation*}
 \int_{{\RR}^d}\Q(f) \, \left(\begin{array}{l}1\\v\\|v|^2\end{array}\right)\,dv = 0,
 \end{equation*}
and it satisfies well-known Boltzmann's $H$ theorem 
 \begin{equation*}
\frac{d H}{dt}(t) \,:=\, -\frac{d}{dt}\int_{{\RR}^d} f \log f \, dv = - \int_{{\RR}^d} \Q(f)\log(f) \, dv \geq 0,
 \end{equation*}
where the functional $H$ is called the {\em entropy} of the solution. Boltzmann's $H$ theorem implies that any equilibrium distribution function, {\em i.e.}, any function which is a maximum of the entropy, has the form of a locally Maxwellian distribution
 \begin{equation*}
%\label{maxw}
 \M[\rho,u,T](v)\,\,=\,\,\frac{\rho}{(2\,\pi\, k_B\;T)^{d/2}}
 \exp \left( - \frac{\vert u - v \vert^2} {2\,k_B\,T} \right), %\label{eq:MAX}
 \end{equation*}
where  $k_B$ is the Boltzmann constant,  $\rho,\,u,\,T$ are the {\em density}, {\em macroscopic velocity}
and {\em temperature} of the gas, defined by
 \begin{equation}
\label{field}
 \rho = \int_{v\in{\RR}^d}f(v)\,dv, \quad u =
 \frac{1}{\rho}\int_{v\in{\RR}^d}v\,f(v)\,dv, \quad T = {1\over{d\rho}}
 \int_{v\in{\RR}^d}\vert u - v \vert^2\,f(v)\,dv.
 \end{equation}
For further details on the physical background and derivation of
the Boltzmann equation we refer to Cercignani, Illner, Pulvirenti~\cite{Cerc}
and Villani~\cite{Vill:hand}.

In order to define completely the mathematical problem for equation 
(\ref{eq:Q}) suitable boundary conditions on $\partial \Omega$ should
be considered. The most simple model for these is due to Maxwell \cite{Maxwell}, in which it is assumed that the fraction $(1 - \alpha)$ of the emerging particles has been reflected elastically at the wall, whereas the remaining fraction $\alpha$ is thermalized and leaves the wall in a Maxwellian distribution. The parameter $\alpha$ is called accommodation coefficient \cite{Cercignani}.

More precisely, we consider equation (\ref{eq:Q}) supplemented with the following boundary conditions for $x\in \partial \Omega$. The smooth boundary $\partial \Omega$ is assumed to have a unit outer normal $n(x)$ at every $x \in \partial \Omega$  and for  $v \cdot n(x) \ge 0$, we assume that at the solid boundary a fraction $\alpha$ of particles is absorbed by the wall and then re-emitted with the velocities corresponding to those in a still gas at the temperature of the solid wall, while the remaining portion $(1-\alpha)$ is perfectly reflected. This is equivalent to impose for the ingoing velocities
\begin{equation}
f(t,x,v)\,\,=\,\,(1-\alpha)\,\R f(t,x,v)\,\,+\,\,\alpha\, \M\, f(t,x,v),\quad x \in\partial\Omega, \quad
v \cdot n(x) \ge 0,
\label{eq:BOU}
\end{equation}
with $0\leq \alpha\leq 1$ and
\begin{equation}
\left\{
\begin{array}{lll}
\displaystyle\R f(t,x,v)&=& f(t,x,v\,-\,2\,(n(x) \cdot v)\,n(x)), 
\\
\,
\\
\displaystyle\M f(t,x,v)&=&\mu(t,x)\; f_w(v).
\end{array}\right.
\label{eq:MAX}
\end{equation}
If we denote by $k_B$ the Boltzmann's constant and by $T_w$ the temperature of the solid boundary, $f_w$ is given by
$$
f_w(v) \,:=\, \exp \left(-\frac{v^2} {2k_BT_w}\right),
$$
and the value of $\mu(t,x)$ is determined by mass conservation at the surface of the wall for any $t\in\R^+$ and $x\in\partial\Omega$
\begin{equation} 
\mu(t,x) \,\int_{v \cdot n(x) \geq 0}f_{w}(v) \,v \cdot n(x)\, dv \,\,=\,\,
-\int_{v \cdot n(x) < 0}f(t,x,v)\,v \cdot n(x) \, dv.
\label{eq:MU}
\end{equation}

Hence, we have
\begin{proposition}
\label{prop:1}
Assume that $f$ is a smooth solution to the Boltzmann equation (\ref{eq:Q}) with boundary conditions (\ref{eq:BOU})-(\ref{eq:MU}). Then we have for any $x\in\partial\Omega$
\begin{equation}
\label{c:1}
\left\{
\begin{array}{l}
\displaystyle\int_{v \cdot n(x) \ge 0}\R f(t,x,v)\,v \cdot n(x) \, dv \,\,=\,\, -\int_{v \cdot n(x) < 0} f(t,x,v)\,v \cdot n(x) \, dv,
\\
\,
\\
\displaystyle\int_{v \cdot n(x) \ge 0}\R f(t,x,v)\,v \cdot \tau(x) \, dv \,\,=\,\, +\int_{v \cdot n(x) < 0} f(t,x,v)\,v \cdot \tau(x) \, dv,
\end{array}\right.
\end{equation}
where $\tau(x)$ belongs to the hyperplane orthogonal to $n(x)$, and 
\begin{equation}
\label{c:2}
\int_{v \cdot n(x) \ge 0}\M f(t,x,v)\,v \cdot n(x) \, dv \,\,=\,\, -\int_{v \cdot n(x) < 0} f(t,x,v)\,v \cdot n(x) \, dv.
\end{equation}
Both equalities (\ref{c:1}) and (\ref{c:2})guarantee the global conservation of mass.
\end{proposition}
\begin{proof}
First the equality (\ref{c:2}) is straightforward by construction of the constant $\mu(t,x)$. 

Then, to prove (\ref{c:1}) for any $x\in\partial\Omega$, we multiply $\R f(t,x,f)$ by a function $\eta(v)$ in (\ref{eq:MAX}) and integrate on the set $\{v\in\RR^d,\,\, (v\cdot n(x)\ge 0 \}$. Applying the change of variable $v^* = v  - 2(v\cdot n(x))\,n(x)$, it yields
$$
\int_{v \cdot n(x) \ge 0}\R f(t,x,v)\,\eta(v) \, dv  = \int_{v^* \cdot n(x) \le 0}f(t,x,v^*)\,\eta\left(v^*\,-\,2(v^*\cdot n(x))\,n(x)\right) \, dv^*.
$$ 
Taking respectively $\eta(v)= v\cdot n(x)$ and $\eta(v)= v\cdot \tau(x)$ we get the result.
\end{proof}
The construction of approximate methods of solution to the nonstationary Boltzmann equation is an important problem in unsteady rarefied flows. The mathematical difficulties related to the structure of the Boltzmann equation make it extremely difficult in most physically relevant situations. For such reasons realistic numerical simulations are based on
Monte-Carlo techniques. The most famous examples are the {Direct
Simulation Monte-Carlo (DSMC)} methods by Bird~\cite{bird} and by
Nanbu~\cite{Na}. These methods guarantee efficiency and
preservation of the main physical properties. However, avoiding
statistical fluctuations in the results becomes extremely
expensive in presence of non-stationary flows or close to
continuum regimes.

More recently a new class of numerical methods based on the use of
spectral techniques in the velocity space has been developed by L. Pareschi \& B. Perthame \cite{PePa:96}. The methods were first derived in~\cite{PePa:96}, inspired from spectral methods in fluid mechanics~\cite{CHQ:88} and by previous works on the use of Fourier transform techniques for the Boltzmann  equation (see~\cite{Boby:88} for instance). The numerical method is based on approximating the distribution function by a periodic function
in the phase space, and on its representation by Fourier series.
The resulting Fourier-Galerkin approximation can be evaluated with
a computational cost of $O(n^{2})$ (where $n$ is the total number
of discretization parameters in velocity), which is lower than
that of previous deterministic methods (but still larger then that
of Monte-Carlo methods). It was further developed by L. Pareschi \& G. Russo in \cite{PaRu:spec:00,PaRu:stab:00} where evolution equations for the Fourier modes were explicitly derived
and spectral accuracy of the method has been proven. Strictly
speaking these methods are not conservative, since they preserve
mass, whereas momentum and energy are approximated with spectral
accuracy. This trade off between accuracy and conservations seems
to be an unavoidable compromise in the development of numerical
schemes for the Boltzmann equation (with the noticeably exception of \cite{Pa:new}).

We recall here that the spectral method has been applied also to
non homogeneous situations~\cite{FiRu:FBE:03, FiRu:04}, to the
Landau equation~\cite{FiPa:02, PaRuTo:00}, where fast algorithms
can be readily derived, and to the case of granular
gases~\cite{NaGiPaTo:03, FiPa:03}. For a recent introduction to
numerical methods for the Boltzmann equation and related kinetic
equations we refer the reader to~\cite{DPR}. Finally let us
mention that A.~Bobylev \& S.~Rjasanow \cite{BoRj:HS:97,BoRj:HS:99}
have also constructed fast algorithms based on a Fourier transform
approximation of the distribution function. In \cite{MP:03}, C.  Mouhot \& L. Pareschi proposed a fast spectral method was proposed for a class of particle interactions including pseudo-Maxwell
molecules (in dimension $2$) and hard spheres (in dimension $3$),
on the basis of the previous spectral method together with a
suitable semi-discretization of the collision operator. This
method permits to reduce the computational cost from $O(n^2)$ to
$O(n\log_2 n)$ without loosing the spectral accuracy, thus making
the method competitive with Monte-Carlo \cite{MP:03,FMP}. The principles and basic
features of this method will be presented in the next sections.

Few works are devoted to the numerical simulation to the Boltzmann equation for a nonhomogeneous gas. We mention for instance the early work of G. Russ and the author using a spectral approximation with a computational cost of $N^2$ , where $N$ is the number of freedom in velocity and a time relaxed scheme \cite{FiRu:FBE:03}. More recenlty, I. Gamba \& S. H. Tharkabhushanam proposed some numerical simulations for a shock tube problem using deterministic problems \cite{GT09, GT10}.

The goal of this paper is to propose an efficient algorithm for the approximation to the time evolution Boltzmann equation in a bounded physical domain $\Omega\subset\RR^d$ supplemented with different boundary conditions. Moreover, we apply a specific  time discretization based on asymptotic preserving scheme allowing to deal rarefied regime as well as its hydrodynamic limit. We will treat several problems for the study of the long time behavior of the solution to the Boltzmann equation : trend to equilibrium, Poiseuille flows, ghost effects.

The plan of the paper is the following. In the next sections we recall the main ingredients for the approximation  of the Boltzmann equation : a Fourier-Galerkin method fot the Boltzmann operator \cite{PaRu:spec:00,FiRu:FBE:03, FiRu:04,MP:03,FMP}, a second order finite volume scheme for the transport \cite{FSB} and finally a stable scheme for the time discretization allowing to treat a wide range of regimes (from rarefied to hydrodynamic) \cite{FJ}.  Section~\ref{sec4} is devoted to numerical and for one and two dimensional, time dependent and stationary problems.  Finally, in the last section we draw conclusions.

\section{General framework for the discretization of the Boltzmann operator} 
\label{sec2}
\setcounter{equation}{0}

We consider the spatially homogeneous Boltzmann operator written in
the following general form
\begin{equation}
\Q(f) = \int_{\mathcal{C}} \mathcal{B}(y,z)\, \big[
f^\prime  f_\star^\prime - f_\star f \big] \,dy \,dz,\quad v\in\mathbb{R}^d,
\label{eq:Qgen}
\end{equation}
with
 \begin{equation*}
 v^\prime   = v + \Theta^\prime(y,z), \qquad
 v^\prime_\star = v + \Theta^\prime_\star(y,z),  \qquad  v_\star = v +
 \Theta_\star(y,z).
 \end{equation*}

In the equations above, $\mathcal{C}$ is some given (unbounded)
domain for $y,z$, and $\Theta$, $\Theta^\prime$, $\Theta^\prime _\star$ are suitable
functions, to be defined later. This general framework emphasizes
the translation invariance property of the collision operator,
which is crucial for the spectral methods. We will be more precise in the
next paragraphs for some changes of variables allowing to reduce
the classical operator (\ref{eq:Q2}) to the form (\ref{eq:Qgen}).

In this section we remind the basic principles leading to Fourier-Galerkin approximation of the Boltzmann operator, the method is based on the following three steps: 
\begin{itemize}
\item[{\bf 1)}] periodized truncations of the Boltzmann collision operator $\Q(f)$,
\item[{\bf 2)}] expansion of the ditribution  function in a truncated Fourier series of degree $N=(n,\ldots,n)\in\NN^d$,
\item[{\bf 3)}] projection of the quadratic operator in the set of trigonometric polynomial of degree $N$.
\end{itemize}

\subsection{Periodized truncations of the Boltzmann collision operator}
\label{sec2-1} 
Any deterministic numerical method requires to work on a {\it bounded} velocity space.  Here, the idea only consists in adding some non physical binary collisions by {\em periodizing} the function and the collision operator. This implies the loss of some local invariants (some non physical collisions are added). Thus the scheme is not conservative anymore, although it still preserves the mass if the periodization is done carefully. However in this way the structural properties of the collision operator are maintained and thus they can be exploited to derive fast algorithms. This periodization is the basis of spectral  methods and we shall discuss below this  non physical truncation (associated with limit conditions) of this velocity space.

Let us consider the space homogeneous Boltzmann equation in a bounded domain in velocity $\mathcal{D}_L = [-L,L]^d$ with $0<L<\infty$. We truncate the integration in $y$ and $z$ in \eqref{eq:Qgen} since periodization would yield infinite result if not: we set $y$ and $z$ to belong to some truncated domain $\mathcal{C}_R \subset \mathcal{C}$ (the parameter $R$ refers to its size and will be defined later).

%For a compactly supported function with support included in $\B_S$, the ball centered at $0$ with radius $S>0$, one has to prescribe suitable relations (depending on the precise change of variable and truncation chosen) between $S$, $R$ and $L$ in order to retain all possible collisions and at the same time prevent intersections of the regions where $f$ is different from zero (this is the so-called {\em dealiasing condition}). 
Then the {\em truncated} collision operator reads
\begin{equation}\label{eq:HOMtruncat}
\Q^R(f) = \int_{\mathcal{C}_R} \mathcal{B}(y,z) \, 
\big( f^\prime_\star\,f^\prime \,-\, f_\star\,f \big) \,dy \, dz,
\end{equation}
for $v \in \mathcal{D}_L$ (the expression for $v \in\RR^d$ is deduced
by periodization).
% The interest of this representation is to preserve the real
% collision kernel and its properties.
By making some changes of variable on $v$, one can easily prove for
the two choices of variables $y,z$ of the next subsections, that for
any function $\varphi$ periodic on $\mathcal{D}_L$ the following weak
form is satisfied:
\begin{equation}
\label{eq:QRweak} 
\int_{\mathcal{D}_L} \Q^R(f) \,\varphi(v) \,dv =
\frac{1}{4}\int_{\mathcal{D}_L}\int_{\mathcal{C}_R}
\mathcal{B}(y,z) \,f_\star\, f \, 
\left( \varphi^\prime_\star + \varphi^\prime - \varphi_\star - \varphi  \right)\, dy \,dz \, dv.
\end{equation}

Now, we use the representation $\Q^R$ to derive spectral methods.

\subsection{Expansion of the distribution function $f$}
\label{sec2-2}
Hereafter, we use just one index to denote the $d$-dimensional sums
with respect to the vector $k=(k_1,..,k_d)\in \ZZ^d$, hence we set
$$
\sum_{k=-N}^N := \sum_{k_1,\dots,k_d=-N}^N.
$$
The approximate function $f_N$ is represented as the truncated Fourier
series
\begin{equation}
f_N(v) = \sum_{k=-N}^N \hat{f}_k \, e^{i \frac{\pi}L k \cdot v},
\label{eq:FU}
\end{equation}
with the Fourier coefficient $\hat{f}_k$ given by
$$
\hat{f}_k = \frac{1}{(2 L)^d}\int_{\mathcal{D}_L} f(v) \, 
e^{-i \frac{\pi}L k \cdot v }\,dv.
$$
In a Fourier-Galerkin method the fundamental unknowns are the
coefficients $\hat{f}_k(t)$, $\,k=-N,\ldots,N$. We obtain a set of
ODEs for the coefficients $\hat{f}_k$ by requiring that the residual
of (\ref{eq:HOMtruncat}) be orthogonal to all trigonometric
polynomials of degree less than $\|N\|_\infty$.  Hence for $k=-N,\ldots,N$
\begin{equation*}
\int_{\mathcal{D}_L}
\left(\frac{\partial f_N}{\partial t} - \Q^R(f_N)
\right)
e^{-i \frac{\pi}L k \cdot v}\,dv = 0.
%\label{eq:VAR}
\end{equation*}
By substituting expression (\ref{eq:FU}) in (\ref{eq:QRweak}) we get
\begin{eqnarray}
\label{qm}
\Q^{R}(f_N) &=& \sum_{l=-N}^N\,\sum_{m=-N}^N \left[ \beta (l,m)  - \beta (m,m) \right]\, 
\hat{f}_l\,\hat{f}_m \, e^{i \frac{\pi}L
(l+m) \cdot v},
\end{eqnarray}
where the so-called {\em kernel modes} $\beta$ are defined by
\begin{equation}
\label{beta}
\beta (l,m) = \int_{\mathcal{C}_R}
\mathcal{B}(y,z)\, e^{i \frac{\pi}L \big(l\cdot \Theta^\prime(y,z) + m\cdot \Theta_\star^\prime(y,z)\big)}  \,dy \,dz.
\end{equation}

\subsection{Projection of the quadratic operator $\Q^R(f_N)$}
\label{sec2-3}
The {\em spectral equation} is the projection of the collision
equation in $\PP_N$, the $(2N + 1)^d$-dimensional vector space of
trigonometric polynomials of degree at most $N$ in each direction,
{\it i.e.},
$$
\frac{\partial f_N}{\partial t} =\mathcal{P}_N\,\Q^R(f_N),
$$
where $\mathcal{P}_N$ denotes the orthogonal projection on $\PP_N$ in
$L^2(\mathcal{D}_L)$.
 A straightforward computation leads to the following set of ordinary differential equations on the Fourier coefficients
\begin{equation*}
\frac{\partial \hat{f}_k}{\partial t}
= \sum_{{l+m=k}\atop{l,m=-N}}^N \left[ \beta(l,m)\,-\, \beta(m,m)\right]\,\hat{f}_{l}\,\hat{f}_m, \quad -N \leq k\leq N. 
\end{equation*}

%%%%%%%%%%%%%%%%%%%%%%%%%%%%%%%%%%%%%%%%%%%%%%%%%%%%%%%%%%%%%%%%%%%%%%%%%%%%%%%%%
\subsection{Application I: the classical spectral method}
\label{sec2-4}
In the classical spectral method \cite{PaRu:spec:00}, a simple change
of variables in (\ref{eq:Q2}) permits to write
\begin{equation}
\Q(f) = \int_{\RR^d} \int_{\ens{S}^{d-1}}
\mathcal{B}^{\rm c} (u, \sigma)\big( f (v^\prime) f (v_\star^\prime)- f (v) f(v_\star) \big)  \,d\sigma\,du,
\label{eq:G}
\end{equation}
with $u = v - v_\star \in \RR^d$, $\sigma \in \ens{S}^{d-1}$, and
\begin{equation}
\left\{
\begin{array}{l}
v^{\prime} = v \,-\, \frac{1}{2}(u-\vert u\vert \sigma ), \vspace{0.3cm} \\
v_\star^{\prime} = v \,-\, \frac{1}{2}(u+\vert u\vert \sigma), \vspace{0.3cm} \\
v_\star = v \,+\,u.
\end{array}
\right.
\label{eq:VV2}
\end{equation}
Then, we set $ \mathcal{C} := \RR^d \times \ens{S}^{d-1}$ and
$$
\Theta^\prime(u,\sigma) :=  - \frac{1}{2}(u-\vert u\vert \sigma ), \quad  
\Theta^\prime _\star(u,\sigma) := - \frac{1}{2}(u+\vert u\vert \sigma), \quad
\Theta_\star(u,\sigma) := u.
$$
Finally the collision kernel $\mathcal{B}^{\rm c}$ is defined by 
\begin{equation}\label{eq:defBclassic}
\mathcal{B}^{\rm c} (u,\sigma) = B\big(|u|,2 (\hat{u}\cdot \sigma)^2 -1 \big),\quad{\rm with }\quad \hat u = \frac{u}{\vert u\vert}.
\end{equation}

Thus, the Boltzmann operator (\ref{eq:G}) is now written in the
form~(\ref{eq:Qgen}).  Therefore, we consider the bounded domain $\mathcal{D}_L = [-L,L ]^d$,  for the distribution $f$, and the bounded domain
$\mathcal{C}_R = \B_R \times \ens{S}^{d-1}$ for some $R>0$.  The
truncated operator reads in this case
\begin{equation*}
\Q^R(f)(v)=\int_{\B_R\times \ens{S}^{d-1}} \mathcal{B}^{\rm c}
(u,\sigma) \,
\big( f(v'_*) f(v') - f(v_*) f(v) \big) \, d\sigma \, du.
\end{equation*}

Then, we apply the spectral algorithm (\ref{qm}) and get the following {\em kernel modes} $\beta^c (l,m)$
\begin{equation*}
\beta^c (l,m) = \int_{\B_R} \int_{\ens{S}^{d-1}} B(|u|, \cos\theta) \, e^{-i \frac{\pi}L \big( u\cdot\frac{(l+m)}{2} - i \vert u\vert\sigma \cdot \frac{(m-l)}{2}\big)} \,d\sigma\,du.
\end{equation*}
We refer to \cite{PaRu:spec:00,FiRu:04} for the explicit computation of Fourier coefficients $\beta^c (l,m)$ in the VHS case where $B$ is given by \eqref{VHSkernel}. 

\subsection{Application II: the fast spectral method}
\label{sec2-5}
Here we shall approximate the collision operator starting from a
representation which conserves more symmetries of the collision
operator when one truncates it in a bounded domain.  This
representation was used in \cite{BoRj:HS:97,ibraRj} to derive finite
differences schemes and it is close to the classical Carleman
representation (cf. \cite{carl}). Hence, the collision operator (\ref{eq:Q2}) can be written
as
\begin{equation}
\label{eq:Qnew}
\Q(f)(v) = \int_{\RR^d\times\RR^d} \mathcal{B}^{\rm f} (y,
z) \,\delta(y \cdot z) \big[ f(v + z) f (v + y)- f (v + y + z) f(v)\big] \, dy \, dz,
\end{equation}
with
$$
\mathcal{B}^{\rm f} (y, z)= 2^{d-1} \, B\left(|y+z|,
  -\frac{y\cdot(y+z)}{|y|\,|y+z|} \right) \,|y+z|^{-(d-2)}.
$$
Thus, the collision operator is now written in the form (\ref{eq:Qgen})
with $\mathcal{C} := \RR^d \times \RR^d$, 
$$
\mathcal{B}(y,z)=\mathcal{B}^{\rm f} (y,z)\,\delta(y\cdot z),
$$
and 
$$
v^\prime_\star = v + \Theta^\prime_\star(y,z),\quad v^\prime = v +
\Theta^\prime(y,z),\quad v_\star=v + \Theta_\star(y,z),
$$
with
$$
\Theta^\prime_\star(y,z) := z,\quad \Theta^\prime(y,z):= y,\quad
\Theta_\star(y,z):= y+z.
$$

Now we consider the bounded domain $\mathcal{D}_L = [-L,L ]^d$, ($0<L
<\infty$) for the distribution $f$, and the bounded domain
$\mathcal{C}_R = \B_R \times \B_R$ for some $R>0$. The (truncated)
operator now reads
\begin{equation}
\label{eq:QRfast}
\Q^R(f)(v)=\int_{\mathcal{C}_R} \mathcal{B}^{\rm f} (y,z)\,\delta(y \cdot z)\,
\big( f(v+z) f(v+y) - f(v+y+z) f(v) \big) \, dy \, dz,
\end{equation}
for $v \in \mathcal{D}_L$. This representation of the collision kernel
yields better decoupling properties between the arguments of the
operator and allows to lower significantly the computation cost of the
method by using the fast Fourier transform (see~\cite{MP:03,FMP}). 

\begin{remark}
Let us make a crucial remark about the choice of $R$. When $f$ has
support included in $\B_S$, $S>0$, it is usual (see
\cite{PaRu:spec:00,MP:03}) to search for the minimal period $L$ (in
order to minimize the computational cost) which prevents interactions
between different periods of $f$ during {\em one} collision process.
From now on, we shall always assume that we can take $L$ and $R$ large
enough such that, when needed, $R \ge \sqrt 2 \, L$. Hence all the
torus is covered (at least once) in the integration of the variables
$(g,\omega)$ or $(y,z)$.
\end{remark}

From now, we can apply the spectral
algorithm (\ref{qm})  to this collision operator and the corresponding kernel modes are given by
$$
\beta^f (l,m) = \int_{y\in B_R}\int_{z\in B_R} {\cal B}^{\rm f}(y,z)\, \delta(y\cdot z) \,
e^{i \frac{\pi}T \, \big( l\cdot y + m\cdot z \big)} \, dy \, dz.
$$
In the sequel we shall focus on $\beta^f$, and one easily checks
that $\beta^f(l,m)$ depends only on $|l|$, $|m|$ and $|l \cdot
m|$.

Now  we consider the case of Maxwellian molecules in dimension $d=2$, and hard spheres in dimension $d=3$ (the most relevant kernel for applications) for which $\mathcal{B}^f$ is constant.
Let us describe the method in dimension $d = 3$  (see \cite{MP:03} for other dimensions and more general interactions).

First we change to spherical coordinates
$$
\beta^f (l,m) =\frac{1}{4} \int_{e \in S^2} \int_{e^\prime \in S^2} \delta(e \cdot e^\prime) \left[ \int_{-R}^{R} |\rho|  \, e^{i \frac{\pi}T \rho (l\cdot e)} \, d\rho\right]\,
\left[\int_{-R}^{R} |\rho^\prime|  \, e^{i \frac{\pi}T \rho^\prime (m \cdot e^\prime)}
\, d\rho^\prime \right] \, de \, de^\prime
$$
and then we integrate first $e^\prime$ on the intersection of the
unit sphere with the plane $e^\perp$
$$
\beta^f(l,m) = \frac{1}{4}\int_{e\in S^2}\phi^3_{R}(l\cdot e) \,
\left[\int_{e^\prime\in S^2 \cap e^\perp}\phi^3_{R}(m \cdot e^\prime) \, de^\prime\right] \, de,
$$
where
$$
\phi^3_{R}(s) = \int_{-R}^R |\rho| e^{i\frac{\pi}T \rho s} \, d\rho.
$$
Thus we get the following decoupling formula with two degrees of freedom
$$
\beta_R(l,m) = \int_{e\in S^2_+}\phi^3_{R}(l\cdot e) \,\psi^3_{R}(\Pi_{e^\perp}(m)) \, de,
$$
where $S^{2}_{+}$ denotes the half-sphere and
$$
\psi^3_{R}(\Pi_{e^\perp}(m)) = \int_0^\pi \phi_{R}^3\big(|\Pi_{e^\perp}(m)| \,  cos\theta\big) \, d\theta,
$$
and $\Pi_{e^\bot}$ is the orthogonal projection on the plane $e^\bot$. Then, we can compute the functions $\phi^3_{R}$ and $\psi_R^3$
$$
\phi^3_R(s) = R^2\big( 2 \Sinc(R s) - \Sinc^2(Rs/2)\big), 
\quad \psi^3_R(s) = \int_0^\pi \phi^3_R \big(s \, cos \theta \big) \, d\theta.
$$
Now the function $e \mapsto \phi_{R}^3(l\cdot
e)\,\psi^3_{R}(\Pi_{e^\perp}(m))$ is periodic on $S_2^+$. Taking
a spherical parametrization $(\theta,\varphi)$ of $e \in S^2_+$
and taking for the set $\mathcal{A}$ uniform grids of respective
size $M_1$ and $M_2$ for $\theta$ and $\varphi$ we get
$$
\beta^f (l,m) \simeq \frac{\pi^2}{M_1\,M_2}\sum_{p,q=0}^{M_1,M_2}\alpha_{p,q}(m) \,\alpha^\prime_{p,q}(l).
$$
\section{Space and time discretization}
\label{sec3}
In this section, we first focus on the discretization of the transport step and omit for sake of clarity the collisional operator and the velocity variable is fixed $v_l=l\Delta v\in\RR^d$, with $l=(l_1,\cdots,l_d)\in\ZZ^d$
\begin{equation}
\label{eq:transp}
\left\{
\begin{array}{l}
\displaystyle\frac{\partial f}{\partial t} \,+\, v\,\cdot\nabla_x f \,=\, 0, \,\, \forall\,(t,x)\,\in\RR^+\times \Omega, 
\\
\,
\\
f(t=0,x,v) = f_0(x,v),\,\, x\in\Omega,\,\, v\in\RR^d.
\end{array}
\right.
\end{equation}
We shall develop the schemes in the context of Finite Volume methods for the approximation of the transport part.  We consider $\T$ a mesh of the space domain $\Omega\subset \RR^d$. For any control volume $T_i\in\T$ we denote by $\N(i)$ the set of the neighbours of $i$. If $j\in\N(i)$, $\sigma_{i,j}$ is the common interface between $T_i$ and $T_j$ and $n_{\sigma_{i,j}}$  is the unit normal vector to $\sigma_{i,j}$ oriented from $T_i$ to $T_j$ and we have $n_{\sigma_{i,j}}=-n_{\sigma_{j,i}}$. Let $m_i$ be the Lebesgue measure of the control volume $T_i$, $\Delta t>0$ be the time step and $t^n=n\,\Delta t$. We integrate the transport equation (\ref{eq:transp}) over the control volume $[t^n,t^{n+1}]\times T_i$ and get
\begin{equation}
\label{schema:vf1}
\left\{
\begin{array}{l}
\displaystyle m_i \frac{f_i^{n+1}(v_l) - f_i^n(v_l)}{\Delta t } + \sum_{j\in\N(i)} \F^n(v_l,\sigma_{i,j})  = 0, \quad T_i\in\T, \,\,n\in\NN,
\\
\,
\\
\displaystyle f_i^0(v_l) = \frac{1}{m_i} \int_{T_i} f_0(x,v) dx, 
\end{array}
\right.
\end{equation}    
where $\F^n(v_l,\sigma_{i,j})$ represents a numerical flux on  $[t^n,t^{n+1}]\times\sigma_{i,j}$  
$$
\F^n(v_l,\sigma_{i,j})  = m(\sigma_{i,j})\, \left( v\cdot n_{\sigma_{i,j}}\right)\, f_{i,j}^n(v_l)
$$
and $f_{i,j}^n$ is an approximation of the edge-based fluxes  between times $t^n$ and $t^{n+1}$. This formula defines a class of finite volume schemes. For instance a second order upwind scheme is obtained by taking 
$$
 f_{i,j}^n(v_l) = 
\left\{
\begin{array}{l}
f_i^n(v_l) \,+\, \delta_{i,j} \,\left(f_{j}^n(v_l)  - f_{i}^n(v_l)\right), \textrm{ if\,\,} \left( v\cdot n_{\sigma_{i,j}}\right) \,>\, 0,
\\
\,
\\
f_{j}^n(v_l)  \,+\, (1-\delta_{i,j})\, \left(f_{i}^n(v_l)  - f_{j}^n(v_l)\right),\textrm{ else,}
\end{array}
\right.
$$
 where $\delta_{i,j}$ defines a slope limiter such that for all $k\in\N(i)$ there exists $0<\beta_{i,k}<1$ satisfying 
$$
\left\{
\begin{array}{l}
\displaystyle 0\leq \sum_{\substack{k\in\N(i),\\ k\neq j}} \beta_{i,k} \leq 1,
\\
\,
\\
\displaystyle\delta_{i,j}\,\left( f_j^n(v_l)- f_i^n(v_l)\right) = \sum_{\substack{k\in\N(i),\\k\neq j}} \beta_{i,k}\left(f_{i}^n(v_l) - f_k^n(v_l) \right).
\end{array}\right.
$$
\subsection{Boundary conditions}
\label{sec3-1}
The most difficult part in the actual implementation of finite volume methods for kinetic equations is that of boundary conditions. We will discuss here the situation of solid walls.

Let $\sigma_i$ be an edge of the control volume $T_i$, with $\sigma_i\in\partial \Omega$ and let us denote by $n(\sigma_i)$ a unit outer vector to the edge $\sigma_i$. We define the boundary solution on $\sigma_i$ by 

\begin{equation}
f_{\sigma_i}^n(v_l) = 
\begin{cases} 
(1-\alpha)\,\R \,f_{\sigma_i}^n(v_l)\,\,+\,\,\alpha \,\M\, f_{\sigma_i}^n(v_l), {\rm\,\, if\,\,} v \cdot n(\sigma_i) \geq 0, 
\\
\,
\\
\displaystyle f_i^n(v_l), {\rm \,\,if\,\,}  v \cdot n(\sigma_i) < 0,
\end{cases} 
\label{eq:BF} 
\end{equation} 
where $\R f_{\sigma_i}^n(v_l)$ and $\M f_{\sigma_i}^n(v_l)$ are defined by (\ref{eq:MAX}). On the one hand, we compute an approximation of the operator describing specular reflection $\R f_{\sigma_i}^n(v_l)$, it gives for $v\cdot n(\sigma_i)\ge 0$,
$$
\R^* f_{\sigma_i}^n(v_l) = f_{\sigma_i}^n(v^*), \quad{\rm with\,\, }  v^* \,=\, v_l - 2(v_l\cdot n(\sigma_i))\,n(\sigma_i). 
$$ 
Unfortunately, the distribution function $f_{\sigma_i}^n(.)$ is not necessarily known on $v^*$ and a picewise linear interpolation is applied to compute the value at $v^*$. Then, to guarantee the local flux conservation we  modify the boundary solution by considering the renormalized boundary solution  for $v_l\cdot n(\sigma_i)\ge 0$
\begin{equation}
\label{xi}
\R f_{\sigma_i}^n(v_l)\,=\, \xi^n(\sigma_i)\,\, \R^* f^n_{\sigma_i}(v_l),
\end{equation}
where the nonnegative constant $\xi^n(\sigma_i)$ is given by
$$
\xi^n(\sigma_i) \, \Delta v^d\,\sum_{v_l\cdot n(\sigma_i)\ge 0} v_l\cdot n(\sigma_i)\,f_{\sigma_i}^n(v_l) \,\,=\,\,  -\Delta v^d\,\sum_{v_l\cdot n(\sigma_i) < 0} v_l\cdot n(\sigma_i)\,f_{i}^n(v_l). 
$$
Clearly $\R \,f_{\sigma_i}^n$ is constructed to guarantee a global zero flux property at the boundary $\sigma_i$ and preserves nonnegativity of the distribution at the boundary.

On the other hand, diffusive boundary conditions are also implemented such that the global flux at the boundary 
\begin{equation}
\label{mu}
\M f_{\sigma_i}^n(v_l)\,=\, \mu^n(\sigma_i) \,\exp \left(-\frac{v_l^2} {2k_BT_w}\right),
\end{equation}
where the constant $\mu^n(\sigma_i)\ge 0$ is computed to ensure the zero flux condition on $\sigma_i\subset\partial\Omega$
$$
\mu^n(\sigma_i) \, \Delta v^d\,\sum_{v_l\cdot n(\sigma_i) \ge 0} v_l\cdot n(\sigma_i)\,\exp \left(-\frac{v_l^2} {2k_BT_w}\right) \,\,=\,\,  -\Delta v^d\,\sum_{v_l\cdot n(\sigma_i)< 0} v_l\cdot n(\sigma_i)\,f_{i}^n(v_l). 
$$
From this construction, we easily prove the following property, which is the analogous result of Proposition~\ref{prop:1}.
\begin{proposition}
\label{prop:2}
The scheme (\ref{schema:vf1}) supplemented with the discrete boundary conditions (\ref{eq:BF})-(\ref{mu}) satisfies for all $n\in\NN$ and $\sigma_i\in\partial\Omega$
$$
\sum_{l}\Delta v^d\, \F^n(v_l,\sigma_{i}) = 0.
$$
Moreover, global mass is preserved over time
$$
\sum_{l}\Delta v^d m_i f_i^n(v_l)\,\,=\,\,\sum_{l}\Delta v^d m_i f_i^0(v_l).
$$
\end{proposition}
\subsection{Stable time discretization}
\label{sec3-2}
 Another difficulty in the numerical resolution of the Boltzmann equation (\ref{eq:Q}) is due to the nonlinear stiff collision (source) terms induced by small mean free or relaxation time. In \cite{FJ}, we propose to penalize the nonlinear collision term
by a BGK-type relaxation term, which can be solved explicitly even if  discretized implicitly in time.  Since the convection term in (\ref{eq:Q}) is not stiff, we will treat it explicitly. The source terms on the right hand side of 
(\ref{eq:Q}) will be handled using the ODE solver in the previous section. For example, if a first order IMEX scheme  is used,
we have 
\begin{equation} 
\label{sch:perturb}
\left\{
\begin{array}{l}
m_i  \displaystyle{\frac{f^{n+1}_i-f^n_i}{\Delta t }  + \sum_{j\in\N(i)} \F^n(v_l,\sigma_{i,j}) \,=\, m_i\frac{\Q(f^n_i) \,-\, \Pf(f^n_i)}{\varepsilon}\,\,+\,\,m_i\,\frac{\Pf(f^{n+1}_i)}{\varepsilon},} 
  \\
  \,
  \\
 \displaystyle f^0_i(v_l)  = \frac{1}{m_i} \int_{T_i} f_0(x,v_l) dx\,.
\end{array}\right.
\end{equation}
 Using the relaxation structure of $\Pf(f)$ given by
$$
\Pf(f) \,\,=\,\, \lambda \, \left(\,\M[\rho,u,T](v) \,-\, f(v) \,\right),
$$
it can be written as
\begin{eqnarray}
\nonumber
m_i\,f_i^{n+1} &=& \frac{\varepsilon}{\varepsilon+\lambda^{n+1}\Delta t} \left( m_i\,f_i^n - \Delta t  \sum_{j\in\N(i)} \F^n(v,\sigma_{i,j}) \right) \,+\, \Delta t\,m_i\,\frac{\Q(f_i^n) \,-\, \Pf(f^n_i)}{\varepsilon+\lambda^{n+1} \Delta t} 
\\
\nonumber
\,
\\
\label{AP-1}
&&+\,\, m_i\;\frac{\lambda^{n+1} \Delta t}{\varepsilon+\lambda^{n+1} \Delta t}\,\M_i^{n+1},
 \end{eqnarray}
where $\lambda^{n}=\lambda[\rho^{n}, T^{n}]$ and $\M^{n}_i$ is the local Maxwellian distribution $\M[\rho^n,u^n,T^n]$ computed from $f_i^{n}$ in the control volume $T_i$. Moreover, a second order IMEX scheme can also be implemented \cite{FJ}.   

Let us mention that a similar approach is proposed by G. Dimarco \& L. Pareschi using exponential Runge-Kutta methods for stiff kinetic equations \cite{DP}.

\section{Numerical tests}
\label{sec4}
In this section, we present a large variety of test cases showing the
effectiveness of our method to get an accurate solution of the Boltzmann
equation. We first give a classical example, which illustrates the property of
the time discretization.

 Finally, we present an interesting result in the 2+2 dimensional phase space proving the
high accuracy of our method, which is able to reproduce small effects (ghost
effects).
%In our space dependent tests we used a 2D model of the Boltzmann equation in  velocity space. Realistic simulations with a full 3D model in velocity space require more powerful computer and faster algorithms for storage and retrieval of the Fourier modes of the collision kernel. 

%%%%%%%%%%%%%%%%%%%%%%%%%%%%%%%%%%%%%%%%%%%%%%%%%%%%%%%%%%%%%%%%%%%%%%%%%%%%%%%%%%%%%%
%%%%                                                                             %%%%%
%%%%                test 1                                                       %%%%%
%%%%                                                                             %%%%%
%%%%%%%%%%%%%%%%%%%%%%%%%%%%%%%%%%%%%%%%%%%%%%%%%%%%%%%%%%%%%%%%%%%%%%%%%%%%%%%%%%%%%%

\subsection{Trend to equilibrium}
\label{sec4-1}
We consider the full Boltzmann equation in dimension $d=2$
$$
\derpar{f}{t} + v \cdot \nabla_x f = \frac{1}{\var} \Q(f), \quad x\in [0,1], v \in \RR^2,
$$
with purely specular reflection at the boundary  in $x$. We first introduce the $(d+ 2)$ scalar fields of density $\rho$, mean velocity $u$ and temperature $T$ defined by (\ref{field}). Whenever
$f(t,x,v)$ is a smooth solution to the Boltzmann equation with
specular boundary conditions, one has the global conservation laws
for mass and energy. Therefore, without loss of generality we shall impose
\begin{eqnarray*}
\int_{[0,1]\times\RR^2} f(t,x,v) \,dx \, dv=1, \quad \int_{[0,1]\times\RR^2} f(t,x,v)\,\frac{|v|^2}{2} \,dx \, dv=1.
\end{eqnarray*}
These conservation laws are then enough to uniquely determine the stationary state of
the Boltzmann equation: the normalized global Maxwellian distribution
\begin{equation}
\label{maxwglob}
\M_g(v) = \frac{1}{2\pi k_B}\,\exp\left(-\frac{|v|^2}{2k_B}\right).
\end{equation}
We shall use the following terminology: a velocity distribution of the
form (\ref{maxwglob}) will be called a {\em Maxwellian distribution}, whereas a distribution of the form
\begin{equation}
\label{maxwloc}
\M_l(x,v) = \frac{\rho(x)}{2\pi k_B T(x)}\,\exp\left(-\frac{|v-u(x)|^2}{2k_B T(x)}\right)
\end{equation}
will be called a {\em local Maxwellian distribution} (in the sense that the constants
$\rho$, $u$ and $T$ appearing there depend on the position $x$). We also define the notion of
{\em relative local entropy} $H_l$, the entropy relative to the local Maxwellian, and the
{\em relative global entropy} $H_g$, the entropy relative to the global Maxwellian distribution, by
$$
\H_l(t) = \int f \, \log\left(\frac{f}{\M_l}\right) \, dx \, dv,
\quad \H_g(t) = \int f \, \log\left(\frac{f}{\M_g}\right) \,  dx \, dv.
$$

Our goal here is to investigate numerically the long-time behavior of the solution $f$.
If $f$ is any reasonable solution of the Boltzmann equation, satisfying certain {\em a priori}
bounds of compactness (in particular, ensuring that no kinetic energy is allowed to leak at
large velocities), then it is possible to prove that $f$ does indeed converge to the global Maxwellian
distribution $\M_g$ as $t$ goes to $+\infty$. %Of course, obtaining these {\em a priori} bounds is extremely difficult; as a matter of fact, they have been established only in the spatially homogeneous situation (which means that the distribution function does not depend on the position variable $x$, see the survey in \cite{Vill:hand}) or in a close-to equilibrium setting, and it still constitutes a famous open problem for spatially inhomogeneous initial data far from equilibrium. 
More recently, Desvillettes and Villani \cite{DV:EB:03}, Guo and Strain \cite{Guo:rate} were interested in the study of
rates of convergence for the full Boltzmann equation.
Roughly speaking in \cite{DV:EB:03}, the authors proved that if the solution to the 
Boltzmann equation is smooth enough and satisfies bounds from below, %of the form
%$$
%\forall \, t \geq 0, \, x \in[0,1], \, v \in\RR^2, \quad  f(t,x,v) \geq K_0\,e^{-A_0\,|v|^{q_0}}
%\quad (A_0,K_0 > 0,\, q_0 \geq 2),
%$$
then (with constructive bounds)
$$
\| f(t) - \M_g \| = O(t^{-\infty}),
$$
which means that the solution converges almost exponentially fast to the global equilibrium
(namely with polynomial rate $O(t^{-r})$ with $r$ as large as wanted).

The solution $f$ to the Boltzmann equation satisfies the formula of additivity of the entropy:
the entropy can be decomposed into the sum of a purely hydrodynamic part, and (by contrast) of a
purely kinetic part. In terms of $H$ functional: one can write
$$
\H_g(t)  = \H_l(t) + \H_h(t),
$$
with the hydrodynamic entropy $H_h$ 
$$
\H_h(t) \,\,=\,\, \int_0^1\rho_l(t,x)\,\log\left(\frac{\rho_l(t,x)}{T_l(t,x)}\right) \, dx.
$$
%In fact, we can also show that
%$$
%\H_l(t) \leq \H_g(t), \quad \forall t \geq 0.
%$$
Moreover in \cite{DV:EB:03}, Desvillettes and Villani
conjectured that time oscillations should occur on the evolution
of the relative local entropy. In fact their proof does not rule
out the possibility that the entropy production undergoes important oscillations in time, and actually most of the 
technical work is caused by this possibility. In \cite{FMP}, the authors investigate the same problem with periodic boundary conditions and justify the oscillation frequency and damping rate using a spectal analysis of the linearized Boltzmann equation (see \cite{ElPi:75}).

Here, we performed simulations on the full Boltzmann equation in a
simplified geometry (one dimension of space, two dimensions of
velocity, with pure specular boundary conditions) for different values of the Knudsen number $\varepsilon>0$ with
the fast spectral method to observe the evolution of the entropy
and to check numerically if such oscillations occur. Clearly this
test is challenging for a numerical method due to the high
accuracy required to capture such oscillating behavior.

Then, we consider an initial datum as a perturbation of the global equilibrium $M_g$
begin{equation}
\begin{equation}
\label{ini2}
f_0(x,v) = \frac{1}{2\pi v_{th}^2} \,(1+ A_0 \, \sin(2\pi \,x))\, 
\left[\exp\left(-\frac{|v-v_0|^2}{2 T_0(x)}\right) + \exp\left(-\frac{|v+v_0|^2}{2 T_0(x)}\right) \right],
\end{equation}
with $v_0=\frac{1}{\sqrt{5}}\,(1,1)$,  the constant $A_0=0.2$ and
$$
T_0(x) = \frac{2}{\sqrt{5}}\,\left( \,1 \,+\, 0.1\,\cos(2\pi x)\,\right), \quad x\in [0,1].
$$

We present the time evolution of the relative entropies $\H_g$, $\H_l$ and $\H_h$ in log scale and observe that initially the entropy is strongly decreasing and when the distribution function becomes close to a local equilibrium, some oscillations appear  for small values of $\var\ll 1$ (see Figure \ref{fig3-1}). The superimposed curves yield the time evolution respectively of the total $\H_g$ functional, of
its kinetic part $\H_l$ and the hydrodynamic entropy $\H_h$.

In Figures \ref{fig3-1} and \ref{fig3-2}, we are indeed able to observe oscillations in the entropy production and in the hydrodynamic entropy, where the strength of the oscillations depends a lot on the parameter $\var$.

\begin{figure}[htbp]
\begin{tabular}{cc}
\includegraphics[width=7.5cm,height=7.5cm]{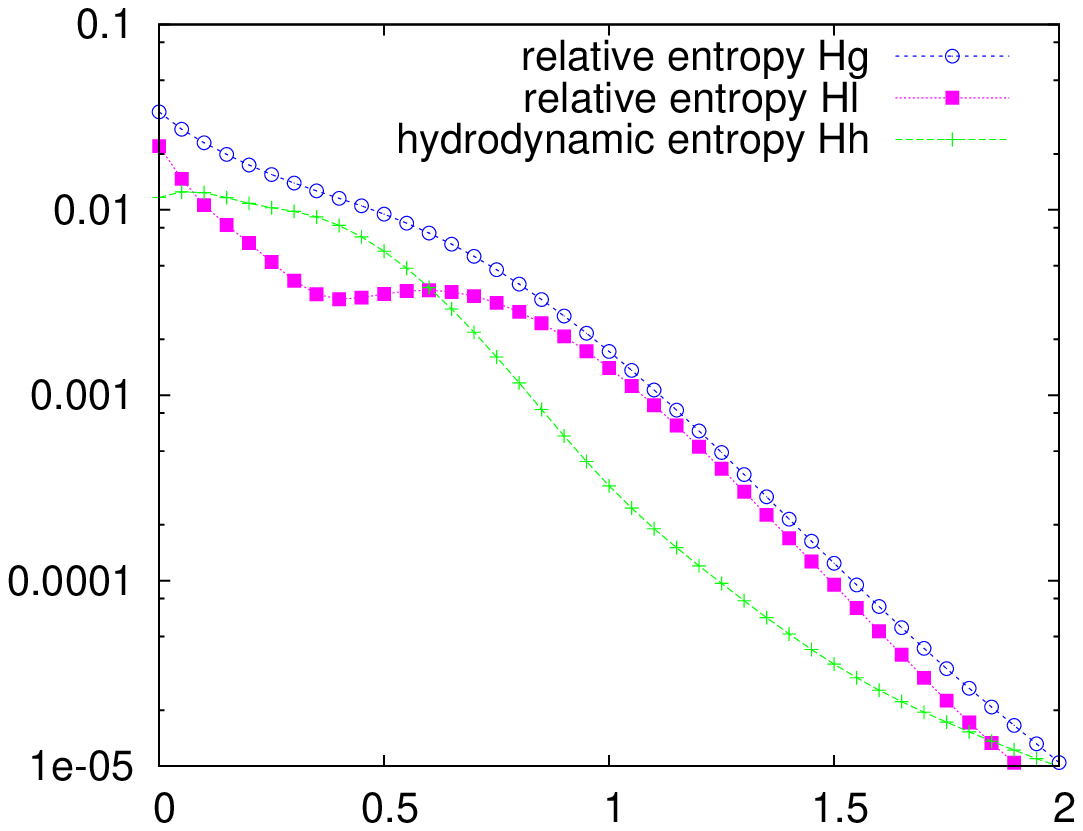} &  
\includegraphics[width=7.5cm,height=7.5cm]{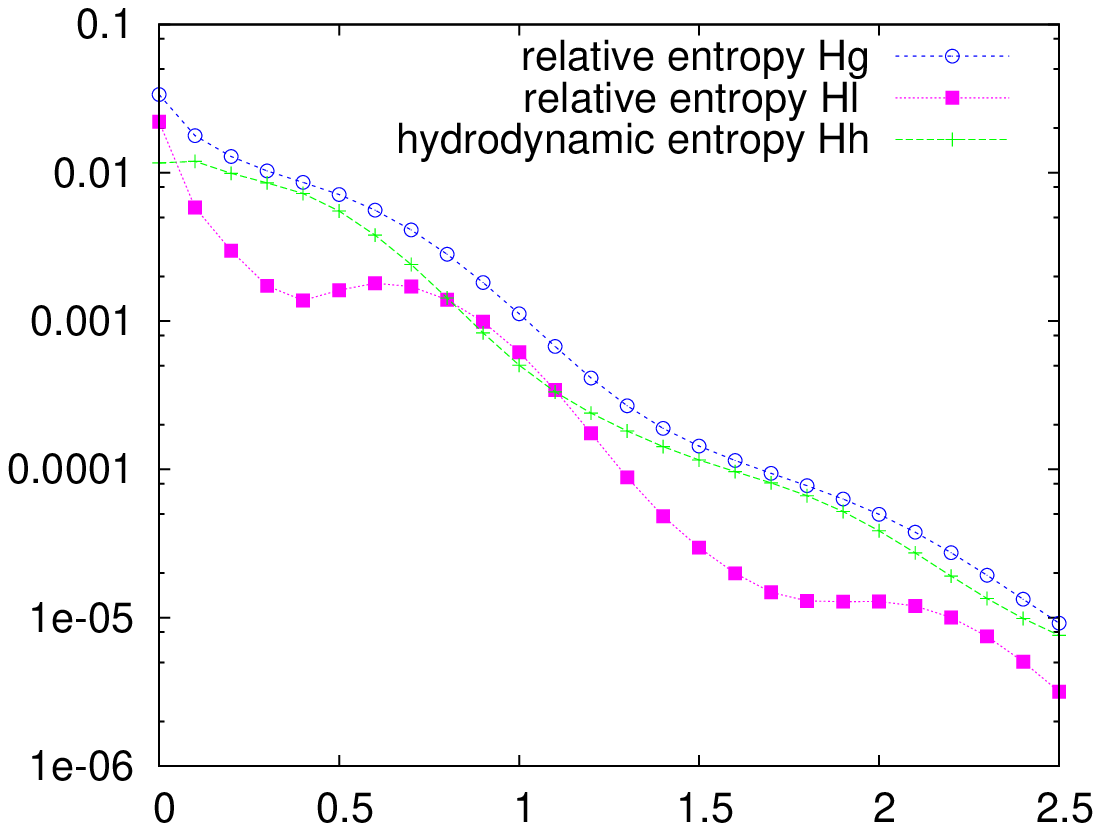} 
\\
(1) $\var=1$ &(2) $\var=0.5$
\end{tabular}
\caption{Trend to equilibrium: {\em time evolution of the entropy relative to the global equilibrium $\H_g$,  entropy relative to the local equilibrium  $\H_l$ and hydrodynamic entropy $\H_h$ (1) $\var=1$ and (2) $\var=0.5$.}}
\label{fig3-1}
\end{figure}

The first plot corresponds to $\var=1$ and the second one to $\var=0.5$. Some slight oscillations can be seen
in the case $\var=1$, but what is most striking is that after a short while, the kinetic entropy is
very close to the total entropy: an indication that the solution evolves basically in a spatially
homogeneous way (contrary to the intuition of the hydrodynamic regime). For  $\var=0.5$, the kinetic entropy $H_l$ is still nonincreasing but after a while it becomes much smaller than the hydrodynamic part and some variations on the decay of the kinetic entropy can be observed.

\begin{figure}[htbp]
\begin{tabular}{cc}
\includegraphics[width=7.5cm,height=7.5cm]{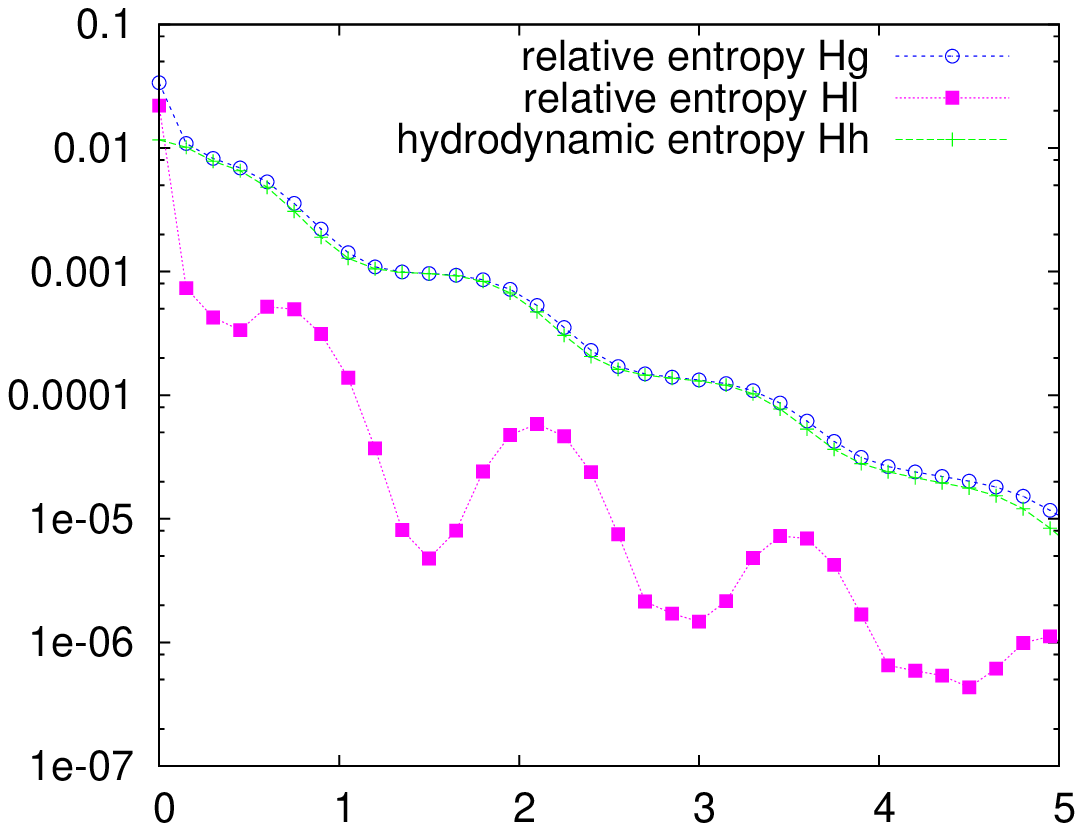} &  
\includegraphics[width=7.5cm,height=7.5cm]{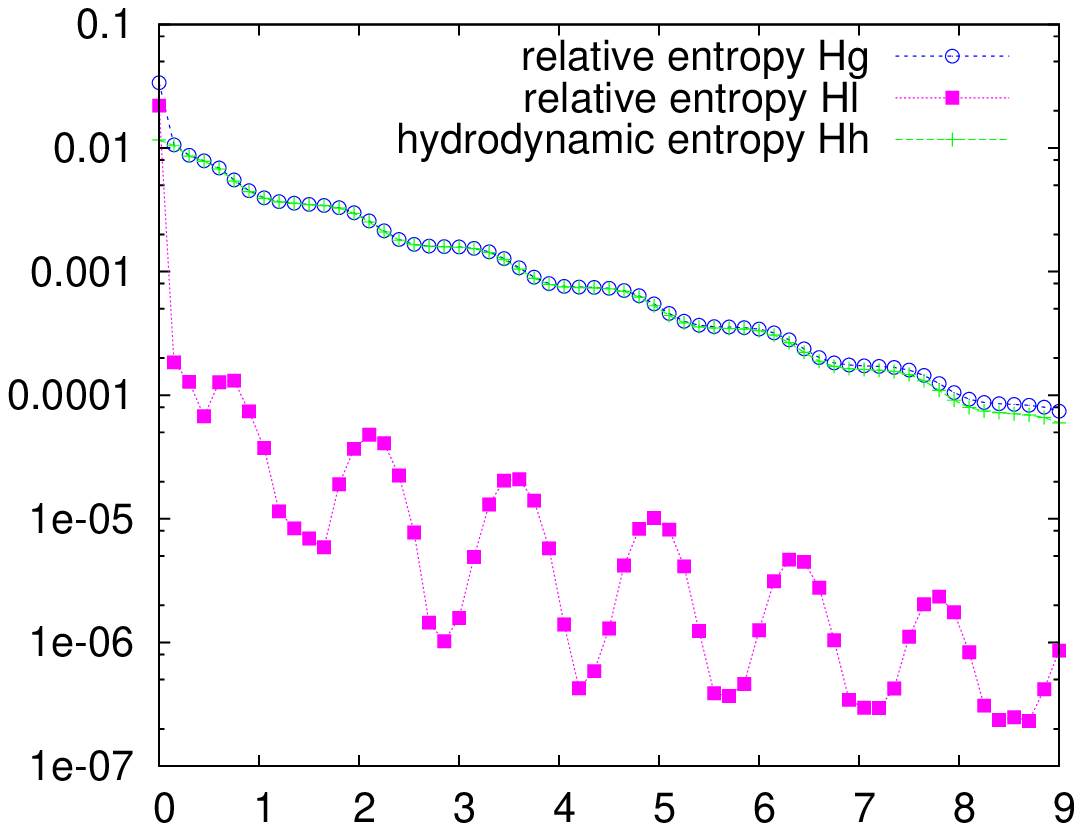} 
\\
(1) $\var=0.15$ &(2) $\var=0.05$ 
\end{tabular}
\caption{Trend to equilibrium: {\em time evolution of the entropy relative to the global equilibrium $\H_g$,  entropy relative to the local equilibrium  $\H_l$ and hydrodynamic entropy $\H_h$ (1) $\var=0.15$ and (2) $\var=0.05$.}}
\label{fig3-2}
\end{figure}

On the contrary, in the case $\var=0.15$ and  more clearly for $\var=0.05$, the oscillations are much more important but they appear with a  small  amplitute and the hydrodynamic entropy is relatively closed to the entropy relative to the global equilibrum. Here we are in the hydrodynamic regime and as it has already be mentionned in \cite{FMP} the oscillation frequency and damping rate of the entropy can be predicted by a precise spectral analysis of the linearized compressible Euler system and compressible Navier-Stokes system.   

Further note that the equilibration is much more rapid when $\var$ is small, and that the convergence seems to be exponential. We also observe that the damping rate is related to $\var$ and is in fact proportional to $\var$ when $\var$ becomes  small.

%%%%%%%%%%%%%%%%%%%%%%%%%%%%%%%%%%%%%%%%%%%%%%%%%%%%%%%%%%%%%%%%%%%%%%%%%%%%%%%%%%%%%%
%%%%                                                                             %%%%%
%%%%                test 1                                                       %%%%%
%%%%                                                                             %%%%%
%%%%%%%%%%%%%%%%%%%%%%%%%%%%%%%%%%%%%%%%%%%%%%%%%%%%%%%%%%%%%%%%%%%%%%%%%%%%%%%%%%%%%%

\subsection{Flow generated by a gradient of temperature}
\label{sec4-2}
We consider the Boltzmann equation (\ref{eq:Q})-(\ref{eq:Q2}) 
$$
\left\{
\begin{array}{l}
\displaystyle \derpar{f}{t} + v_x \frac{\partial f}{\partial x} \,=\, \frac{1}{\varepsilon}\,\Q(f),\,\, x\in (-1/2,1/2),\,v\in\RR^2,
\\
\,
\\
\displaystyle f(t=0,x,v) = \frac{1}{2\pi \,k_B\, T_0(x)} \,\exp\left(-\frac{|v|^2}{2 k_B\,T_0(x)}\right),  
\end{array}\right.
$$
with $k_B=1$, $T_0(x)= 1 \,+\, 0.44\,(x-1/2)$ and we assume purely diffusive boundary conditions on $x=-1/2$ and $x=1/2$, which can be written as 
$$
f(t,x,v) = \mu(t,x)\; f_w(v), \,{\rm if }\,(x,v_x)\in\{-1/2\}\times\RR^+\,{\rm and }\,(x,v_x)\in\{1/2\}\times\RR^-,
$$
where $\mu$ is given by (\ref{eq:MU}). This problem has already been studied in \cite{wagner} using DMCS for the Boltzmann equation or using deterministic approximation usinga BGK model for the Boltzmann equation in \cite{Aoki94}. Here we apply our deterministic scheme and choose a computational domain $[-8,8]\times [-8,8]$ in the velocity space with a number grid points  $n_v=32$ in each direction whereas for the space direction we take $n_x=120$ and the time step $\Delta t=0.001$.

In Figure~\ref{Fig_bl_1}, we represent the stationary solution obtained at $t=25$ of the density, temperature. We also plot the pressure profile of the steady state. The results are in a qualitative good agreement with those already obtained in \cite{wagner} with DSMC. More precisely, the boundary layer (Knudsen layer) appears in the density and temperature as well as the pressure, but it is small for all the quantities. The magnitude in the dimensionless density, temperature, and pressure is of order of $\var$ and the thickness of the layer is, say $O(3 \var)$. In the density and temperature profiles, we cannot observe it unless we magnify the profile in the vicinity of the boundary. Instead, since the pressure is almost constant in the bulk of the gas, we can observe perfectly the boundary layer by magnifying the entire profile. Let us emphasize that, as it is shown in Figure~\ref{Fig_bl_1} the Knudsen layer is a kinetic effect, which disappears in the fluid limit ($\var \rightarrow 0$).

\begin{figure}[htbp]
\begin{tabular}{ccc}
\includegraphics[width=4.125cm,height=4.5cm]{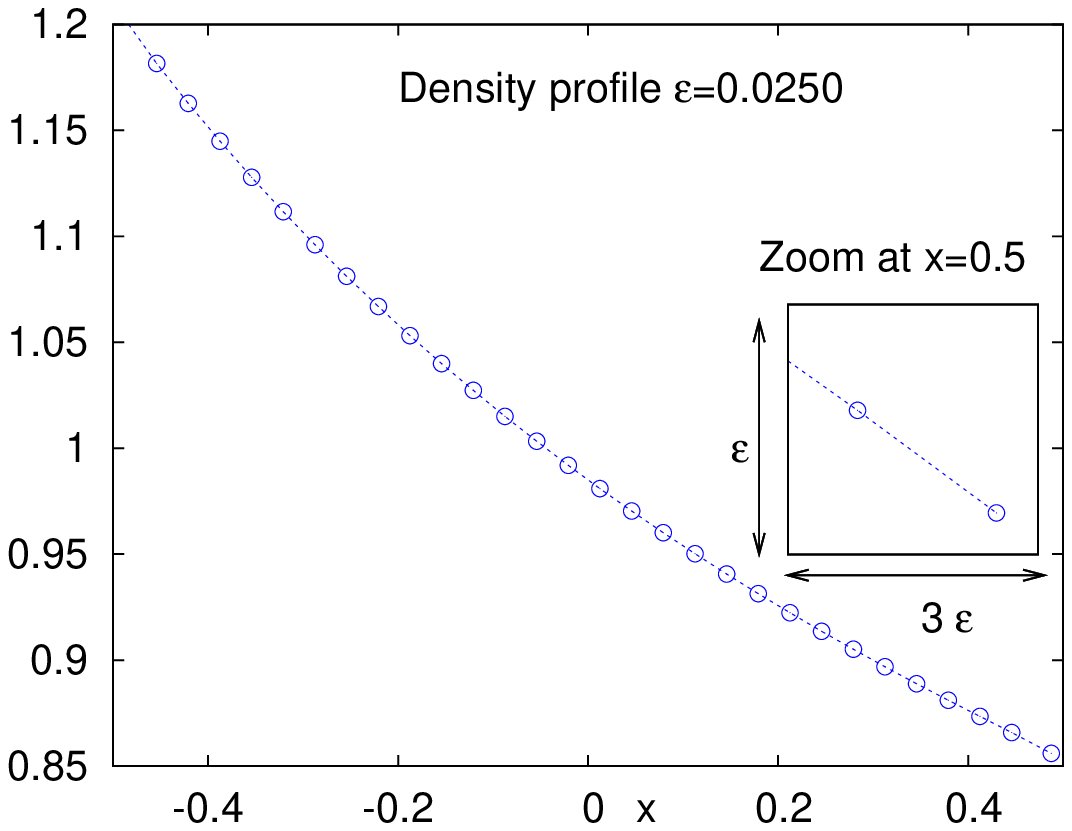} &  
\includegraphics[width=4.125cm,height=4.5cm]{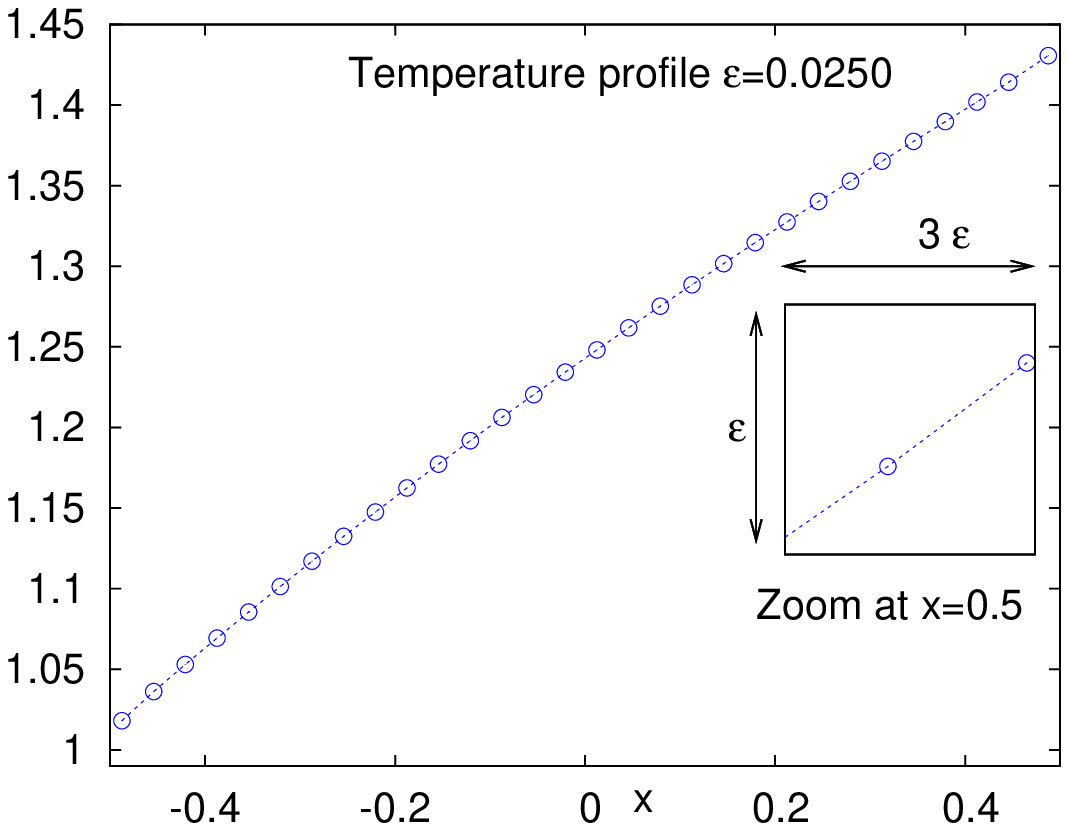} &
\includegraphics[width=6.250cm,height=4.5cm]{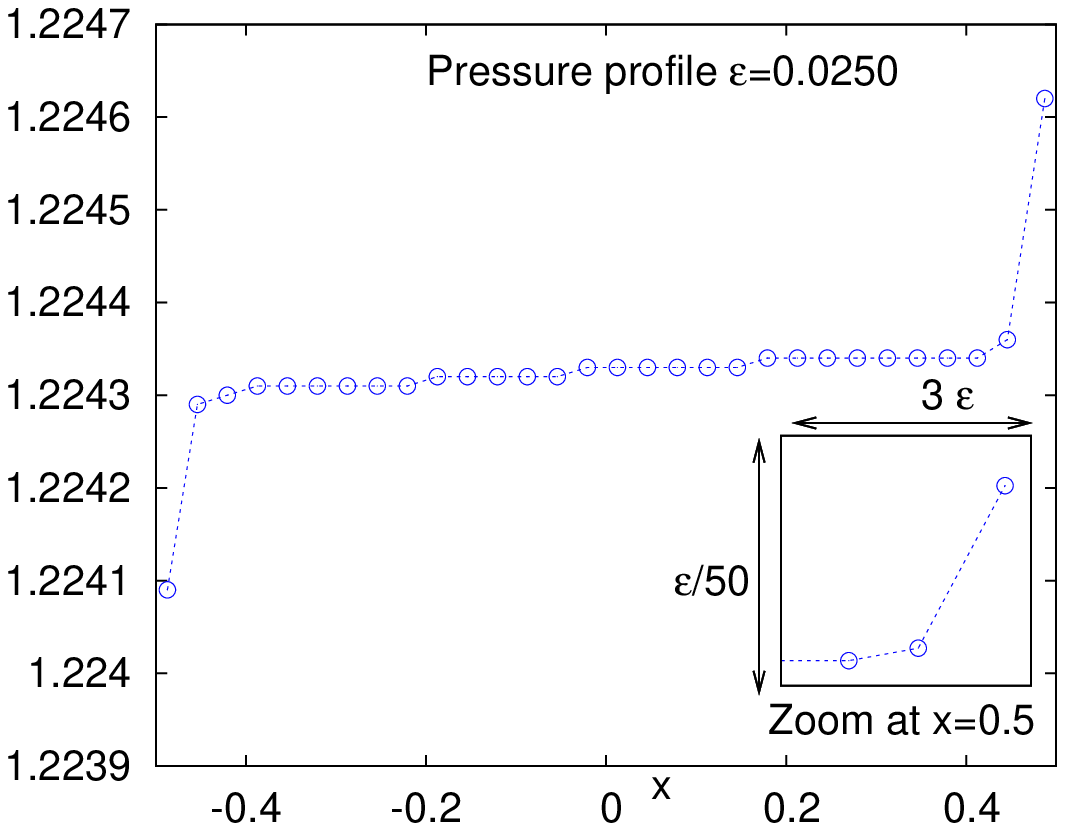} 
\\
\includegraphics[width=4.125cm,height=4.5cm]{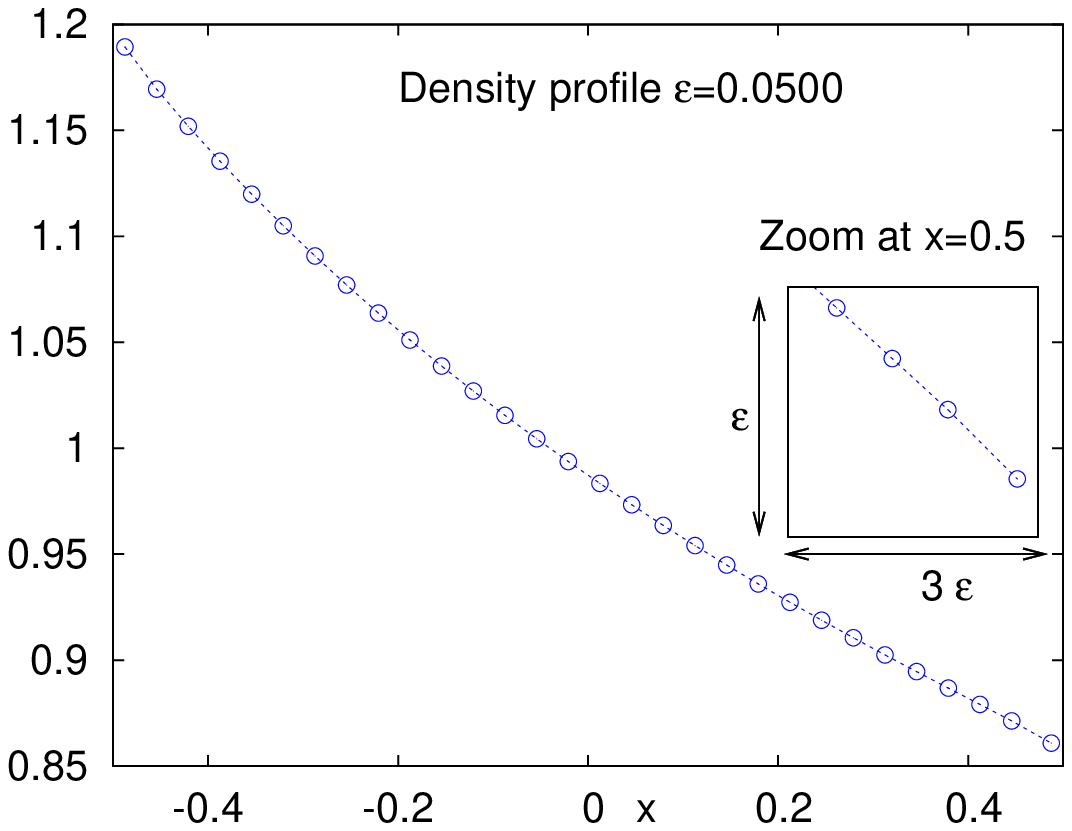} &  
\includegraphics[width=4.125cm,height=4.5cm]{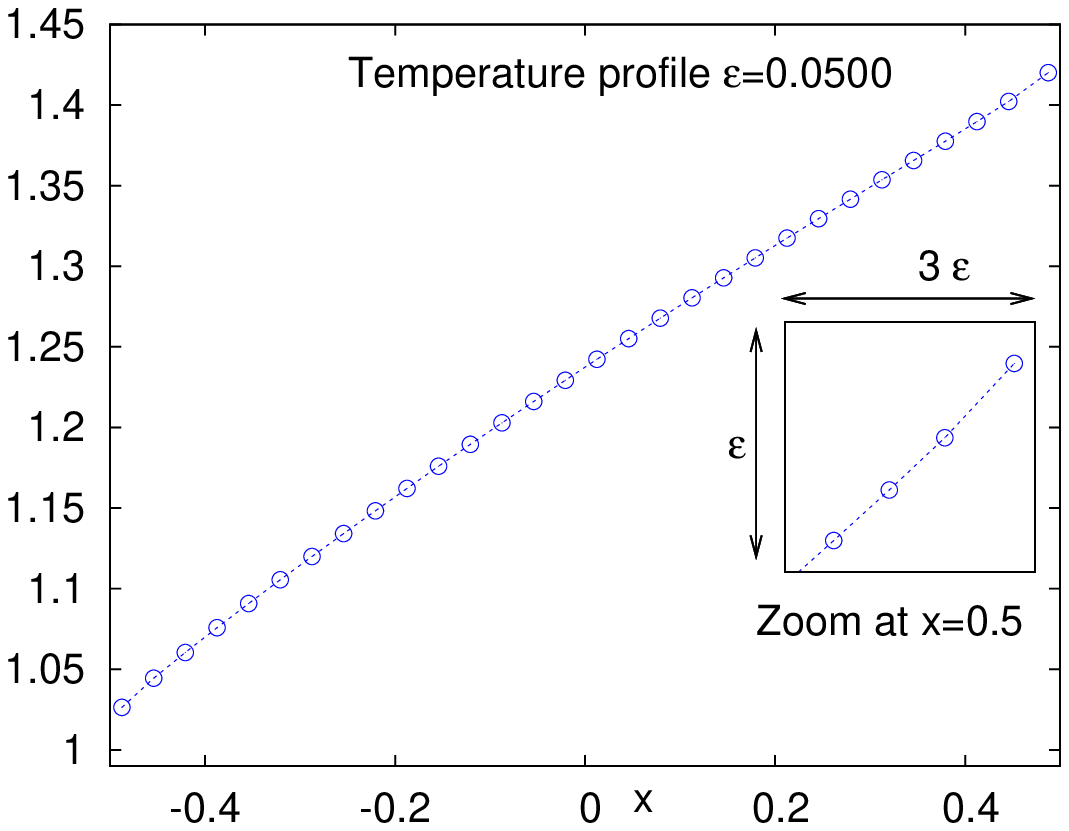} &
\includegraphics[width=6.250cm,height=4.5cm]{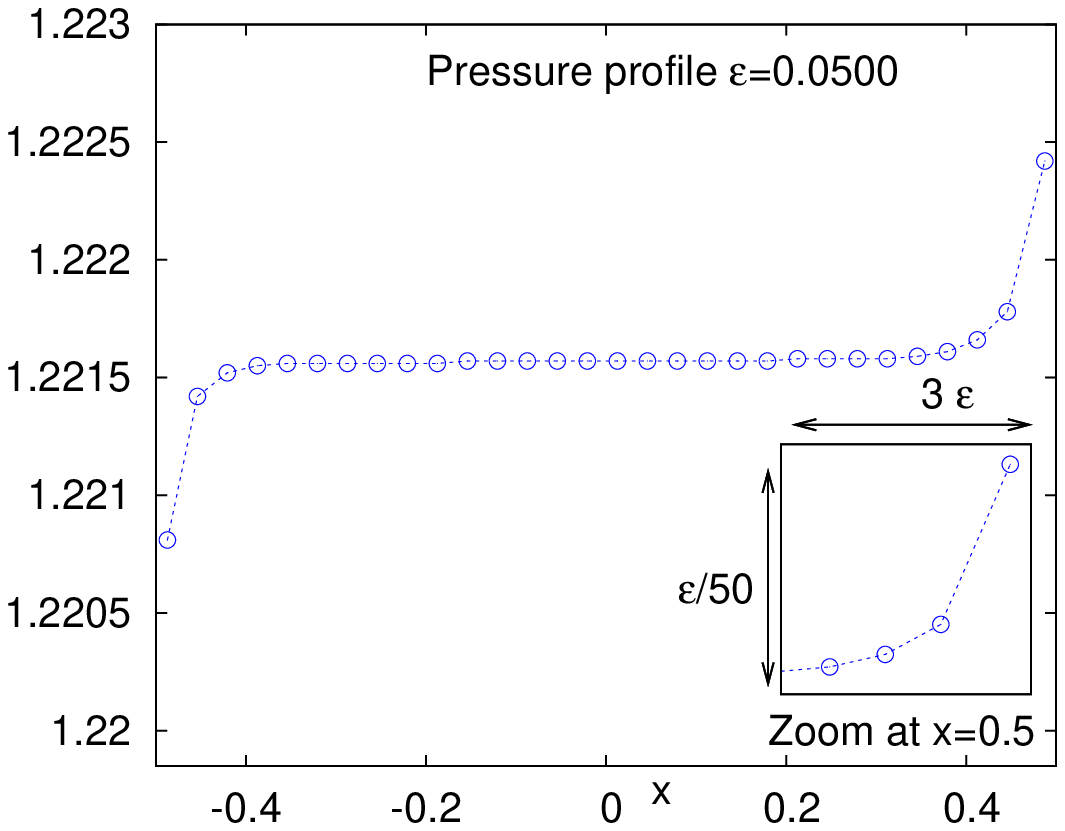}
\\
\includegraphics[width=4.125cm,height=4.5cm]{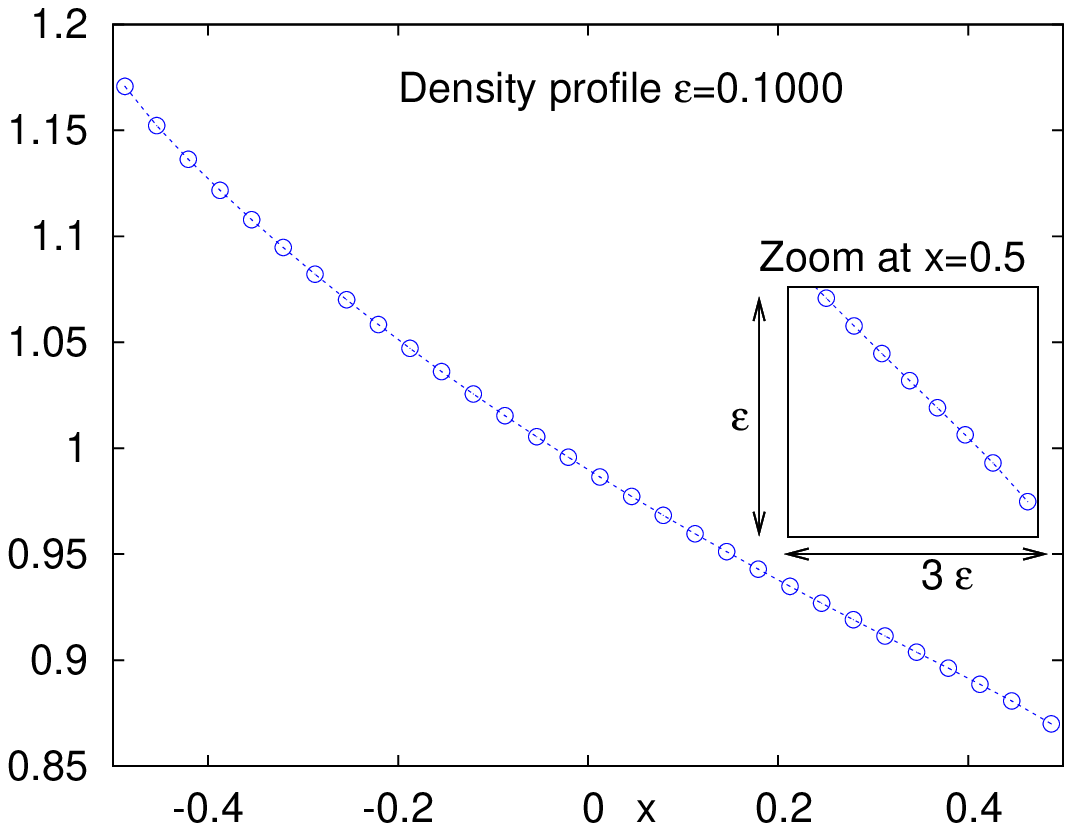} &  
\includegraphics[width=4.125cm,height=4.5cm]{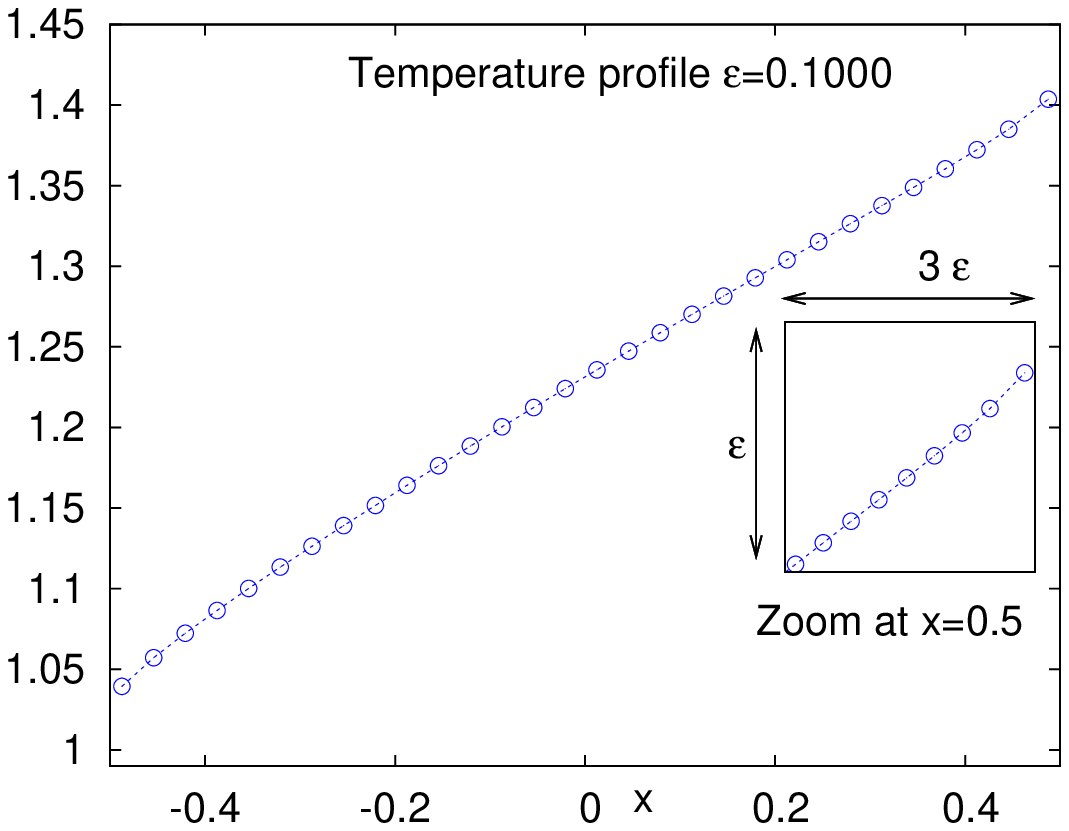} &
\includegraphics[width=6.250cm,height=4.5cm]{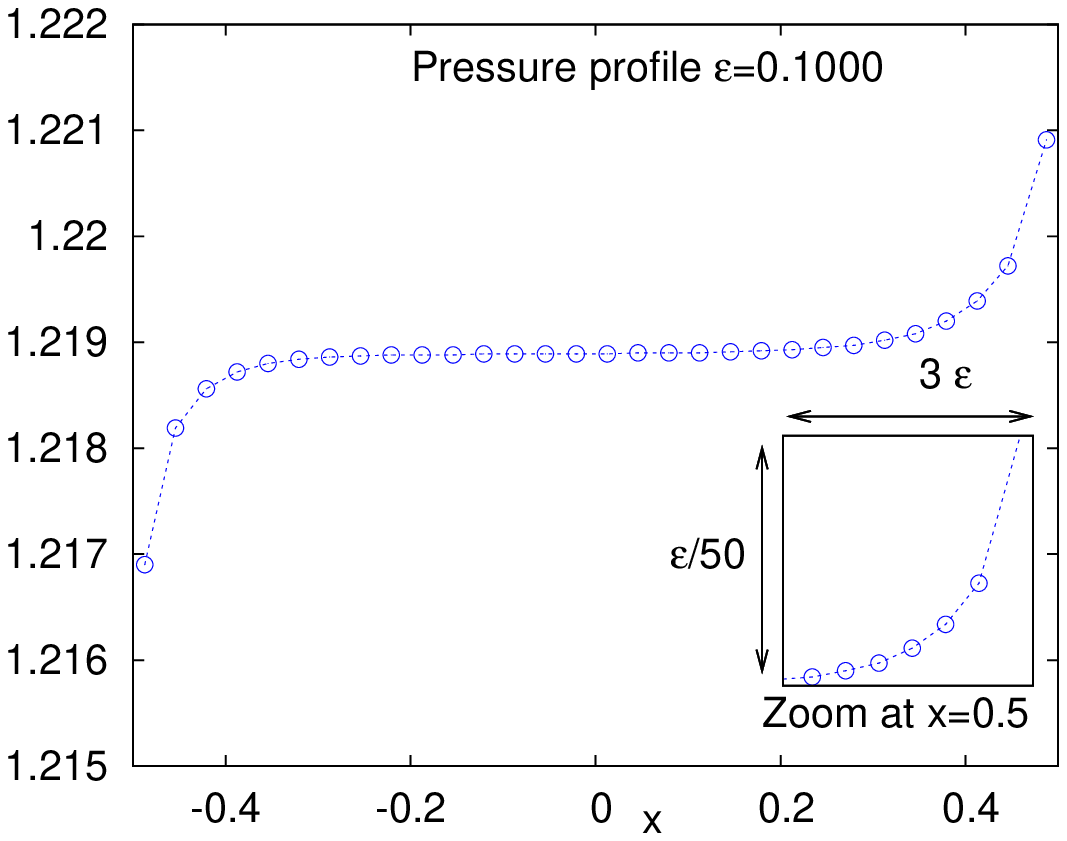}     
\\
(1)&(2)&(3)
\end{tabular}
\caption{Flow generated by a gradient of temperature: {\em (1) density and (2) temperature (3)  pressure   for various Knudsen numbers $\var = 0.025$, $0.05$ and $0.1$.}}
\label{Fig_bl_1}
\end{figure}

These results provide strong evidence that the present deterministic method can be used to determine the state of a gas under highly nonequilibrium conditions. Using deterministic methods, we can investigate the behavior of gases for situations in which molecular diffusion is important e.g., thermal diffusion.

%%%%%%%%%%%%%%%%%%%%%%%%%%%%%%%%%%%%%%%%%%%%%%%%%%%%%%%%%%%%%%%%%%%%%%%%%%%%%%%%%%%%%%
%%%%                                                                             %%%%%
%%%%                test 3                                                       %%%%%
%%%%                                                                             %%%%%
%%%%%%%%%%%%%%%%%%%%%%%%%%%%%%%%%%%%%%%%%%%%%%%%%%%%%%%%%%%%%%%%%%%%%%%%%%%%%%%%%%%%%%

\subsection{Poiseuille-type flow driven by a uniform external force}
\label{sec4-3}
The Poiseuille flow is a classical example to study by means of the Navier-Stokes equations. Usually the Poiseuille flow is understood to be driven by an externally imposed pressure gradient, but it is trivially equivalent to applying a gravitational force over each particle. For small velocities (small Reynolds or Mach number) the flow is known to be laminar and stationary and the velocity profile is parabolic. There is, however, a critical Knudsen number above which an unstable regime starts and the flow can be described using Burnett equations.

Here, we consider an ideal gas between two parallel infinite plates at rest with a common uniform temperature. When the gas is subject to a uniform external force in the direction parallel to the plates, a steady unidirectional flow of the gas is caused between the plates. Assume that the plates are at rest and located at $x=0$ and $x=1$ and kept at temperature	$T_w=1$. The gas is subject to a uniform external force in the $y$-direction, {\it i.e.}, in the direction parallel to the plates. There is no pressure gradient in the $y$-direction. We investigate the steady flow of the gas caused by the external force on the basis of kinetic theory for a wide range of the Knudsen number, paying special attention to the behavior for small Knudsen numbers. Our basic assumptions are as follows:
\begin{itemize}
\item the behavior of the gas is described by the Boltmzann model (\ref{eq:Q})-(\ref{eq:Q2}), 
\item the gas molecules are reflected diffusely on the plates.
\end{itemize}
The Boltzmann equation in the present problem is written as
$$
\left\{
\begin{array}{l}
\displaystyle \derpar{f}{t} + v_x \frac{\partial f}{\partial x} \,+\, a \frac{\partial f}{\partial v_y} \,=\, \frac{1}{\varepsilon}\,\Q(f),
\\
\,
\\
f(t=0) = f_0,
\end{array}\right.
$$
with purely diffusive boundary conditions. In this section, we give some numerical results for intermediate values of the Knudsen number in the case where $a$ is fixed. We apply our determinitic scheme and choose a computational domain $[-8,8]\times [-8,8]$ in the velocity space with a number grid points  $n_v=32$ in each direction and for the space direction $n_x=64$. Finally we take $\Delta t=0.002$.

In Figures~\ref{Fig_pf_1}, we show the profiles of the density, flow velocity, and temperature for various Knudsen numbers $\varepsilon>0$ but for a fixed value of the force parameter $a=0.5$. In that case, the driven force amplitude $a$ is small and then the flow speed is naturally low, and thus the density and temperature profiles are nonuniform.

We do not present the plot of the pressure, which is not uniform but its nonuniformity is quite small since it is proportional to $\epsilon^2$. The most interesting remark is that the temperature profile is not parabolic and its nonuniformity is dominated by a $x^4$ term. The temperature at this order has a minimum at the center of the channel but it has symmetric maxima quite near the center as it has been already shown in \cite{santos}

\begin{figure}[htbp]
\begin{tabular}{cc}
\includegraphics[width=7.5cm]{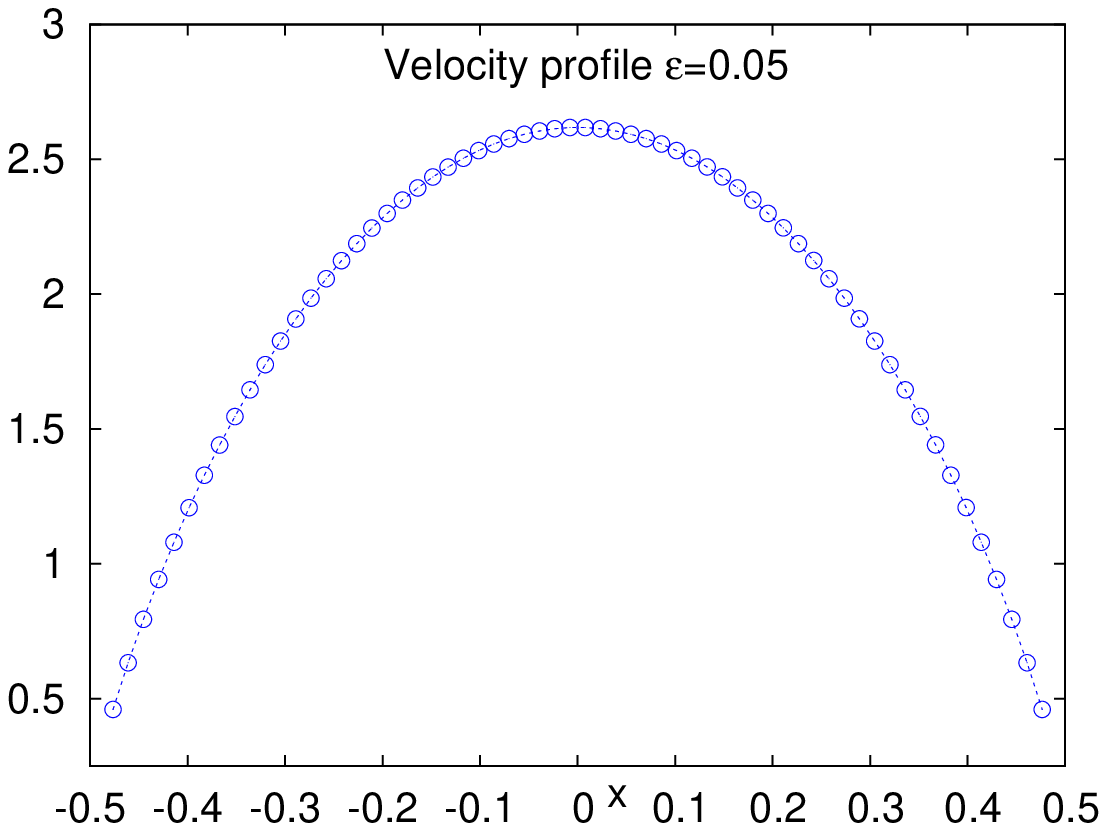} &  
\includegraphics[width=7.5cm]{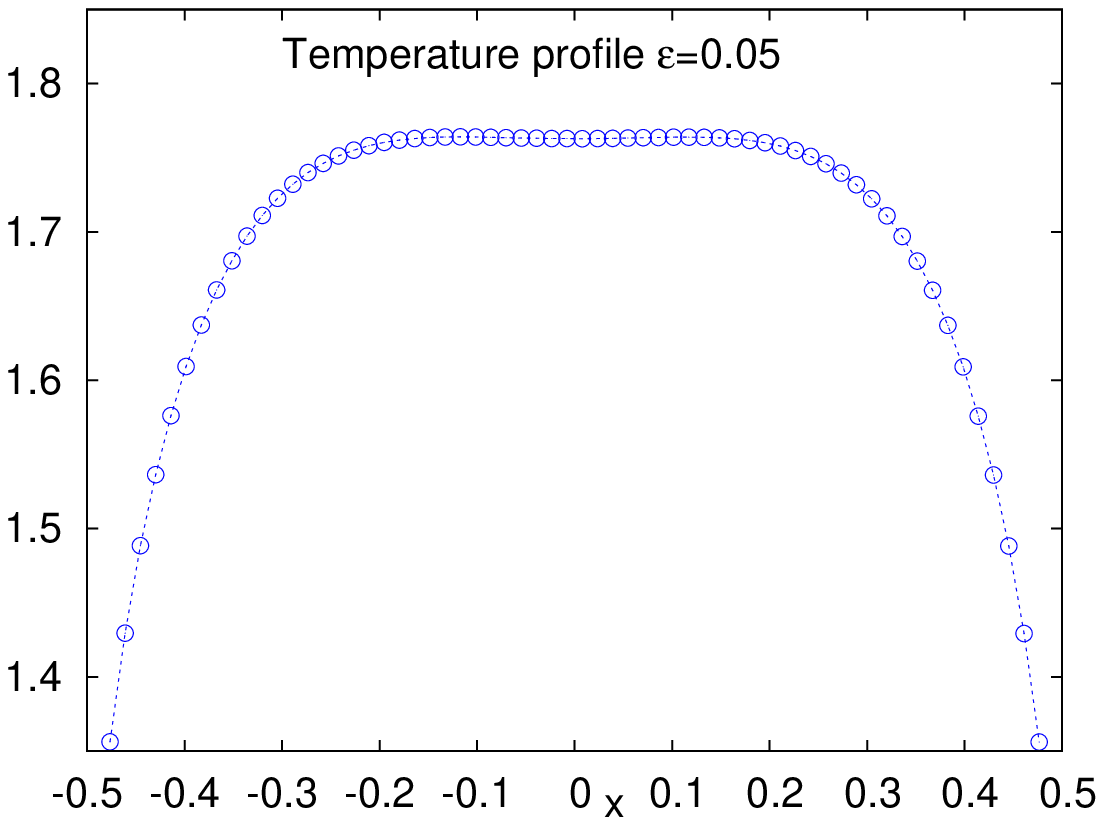}     
\\
\includegraphics[width=7.5cm]{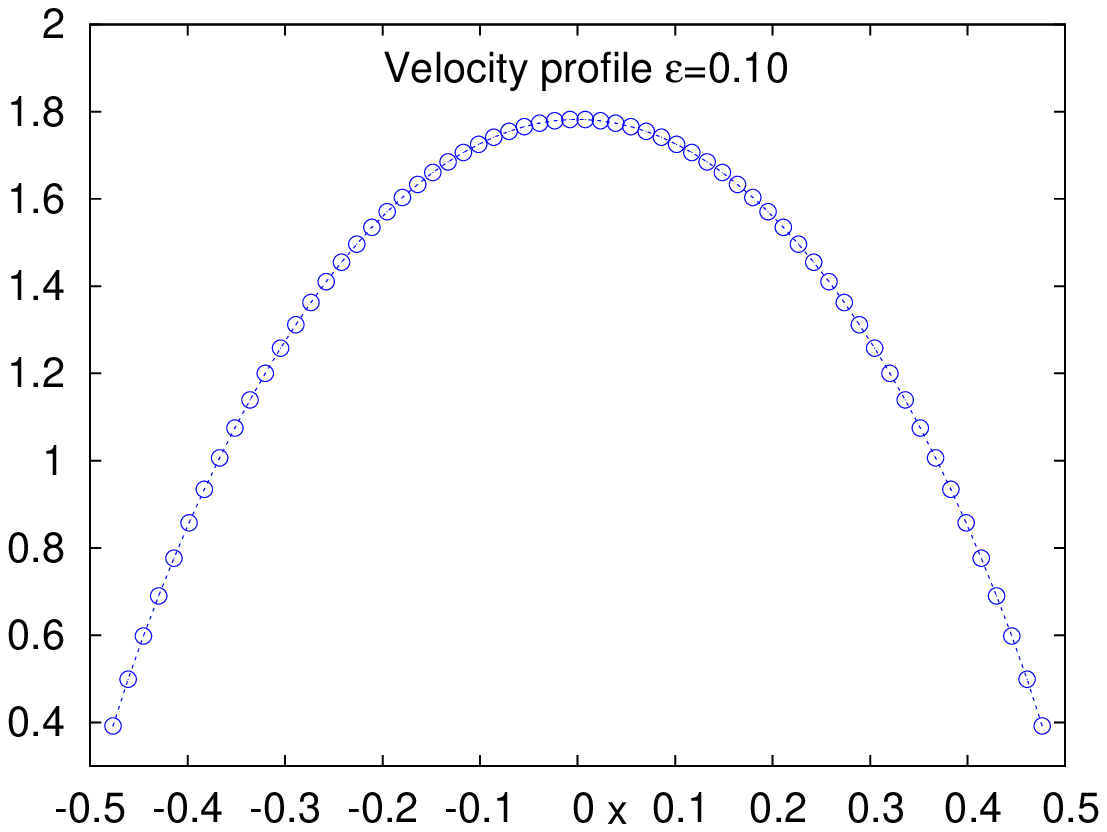} &  
\includegraphics[width=7.5cm]{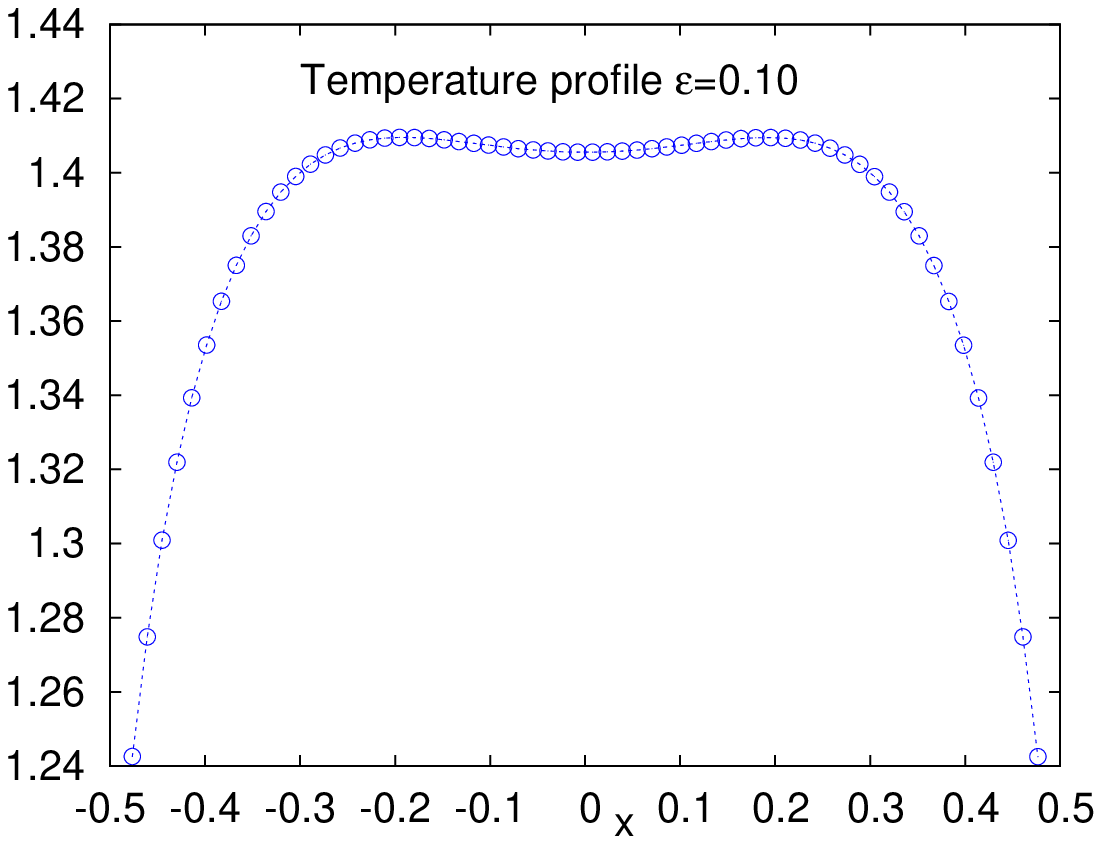} 
\\
\includegraphics[width=7.5cm]{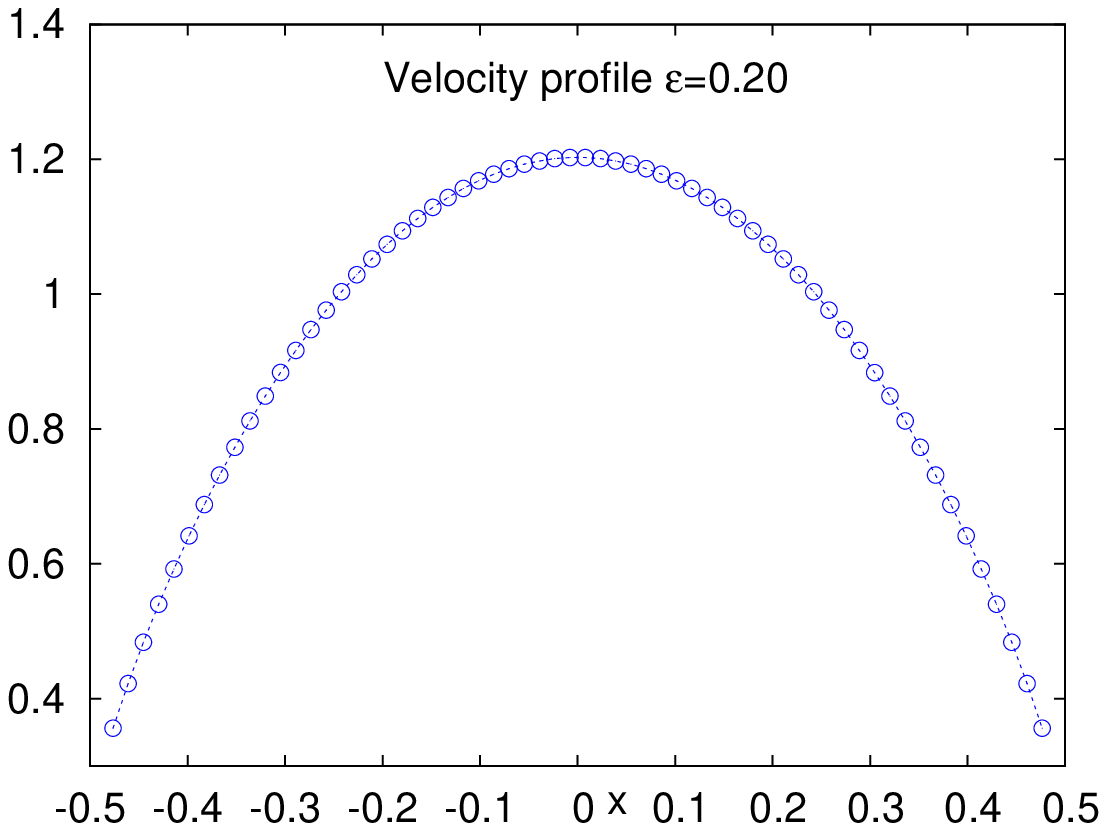} &  
\includegraphics[width=7.5cm]{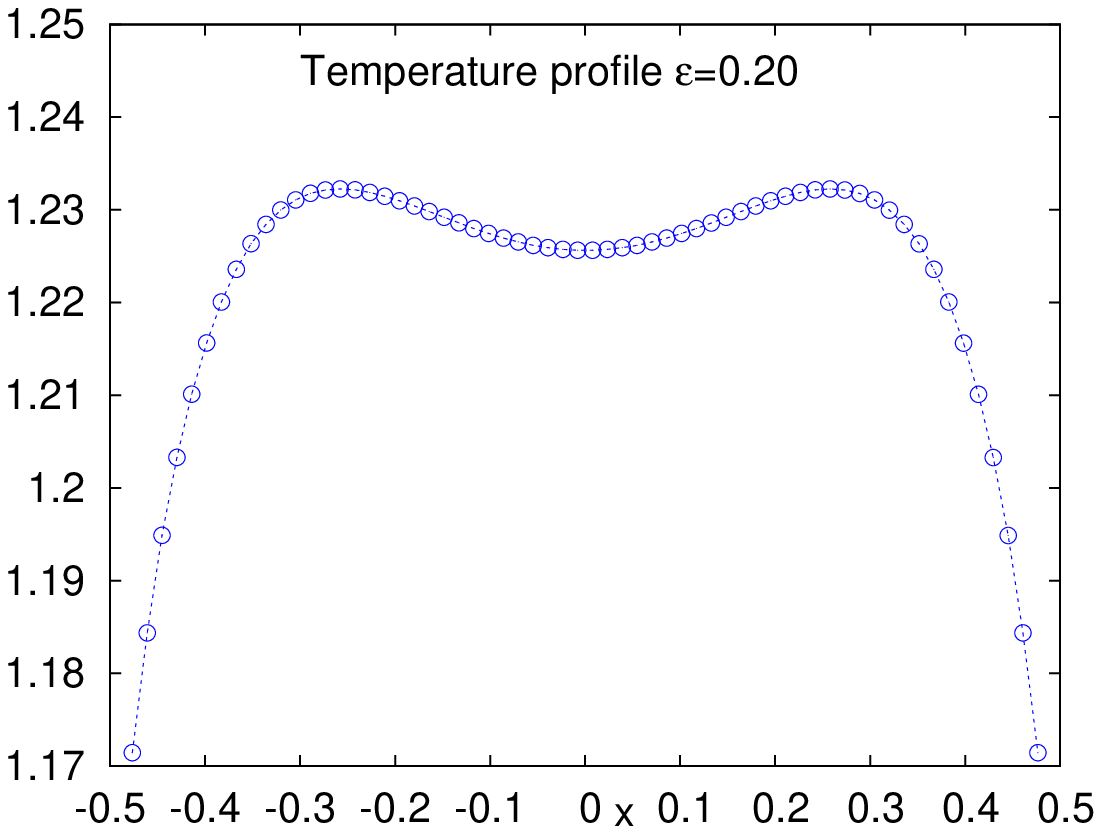} 
\\
\includegraphics[width=7.5cm]{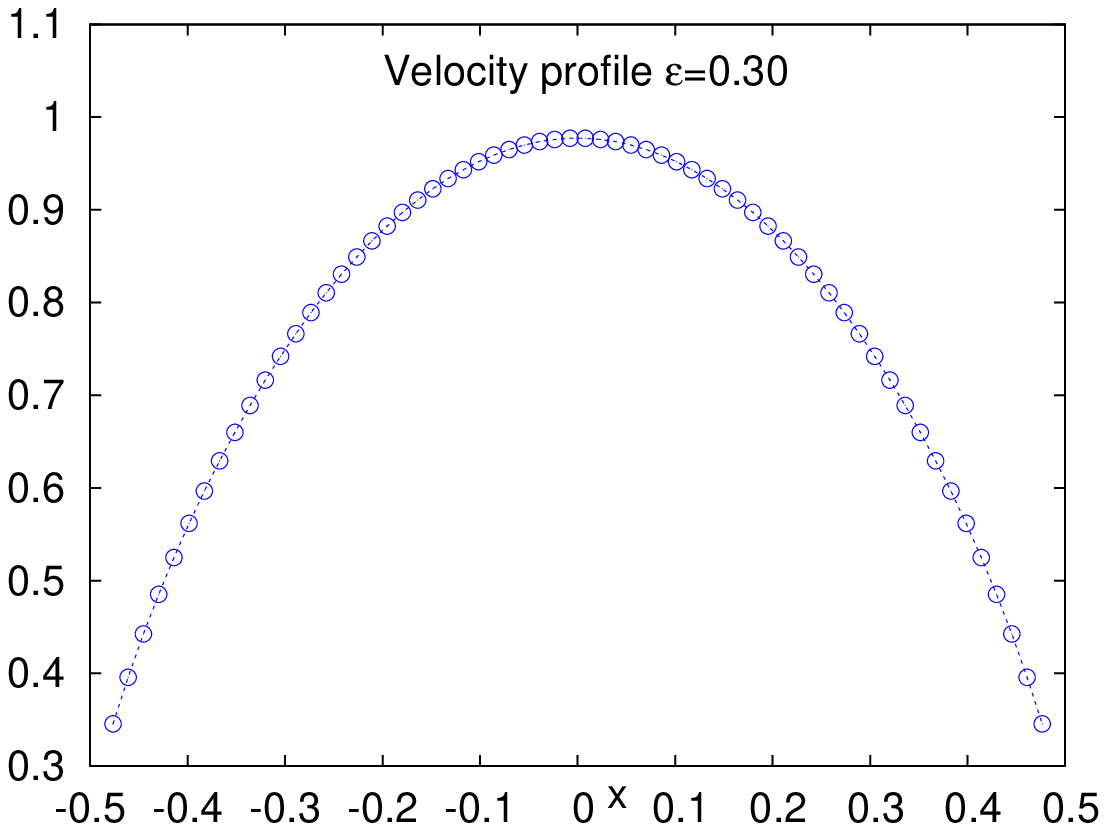} &  
\includegraphics[width=7.5cm]{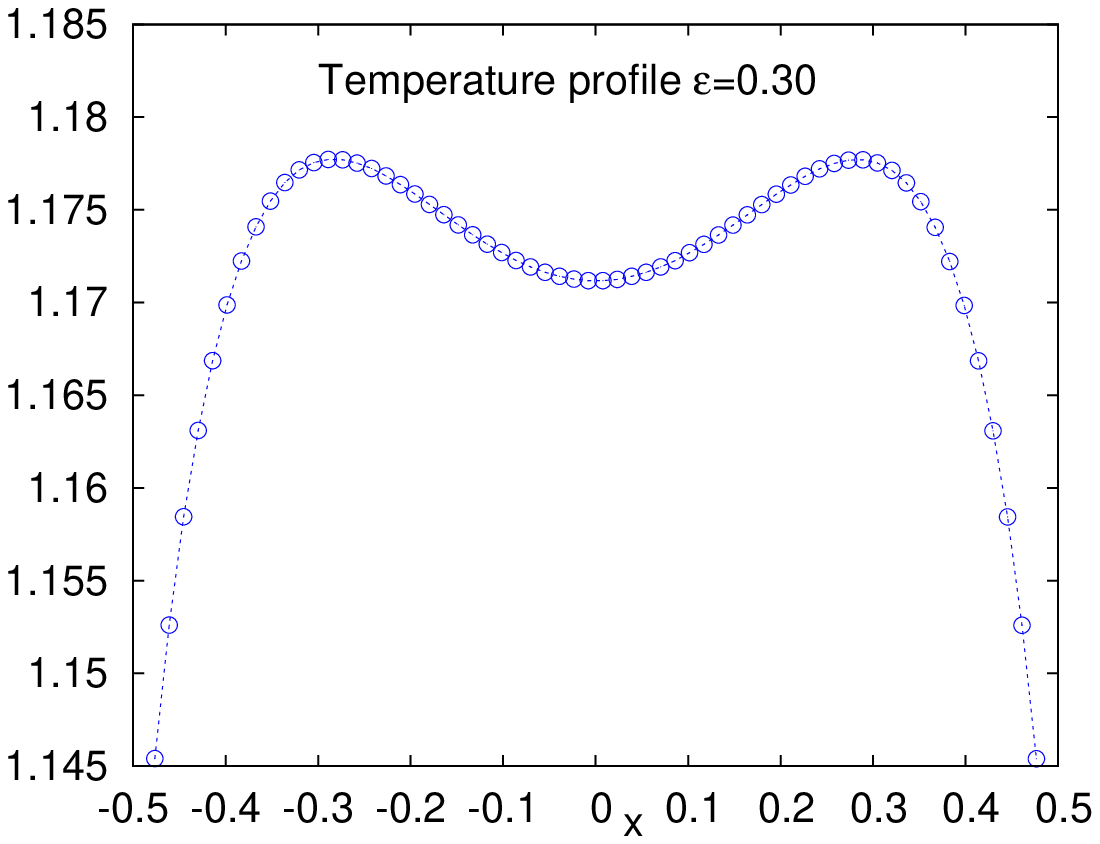} 
\\
(1)&(2)
\end{tabular}
\caption{Poiseuille flow: {\em (1) mean velocity and (2) temperature  for various Knudsen numbers $\var=0.05$, $0.1$, $0.2$ and $0.3$.}}
\label{Fig_pf_1}
\end{figure}

\begin{figure}[htbp]
\begin{tabular}{cc}
\includegraphics[width=7.5cm]{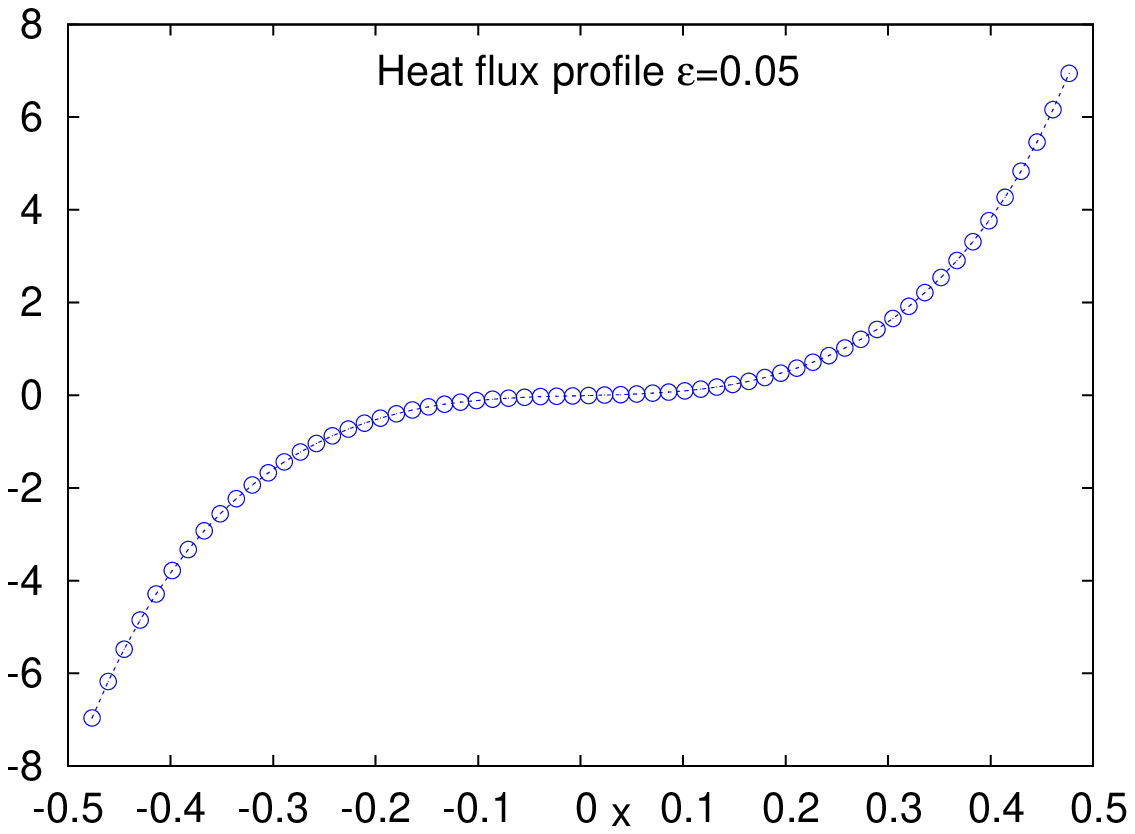} &
\includegraphics[width=7.5cm]{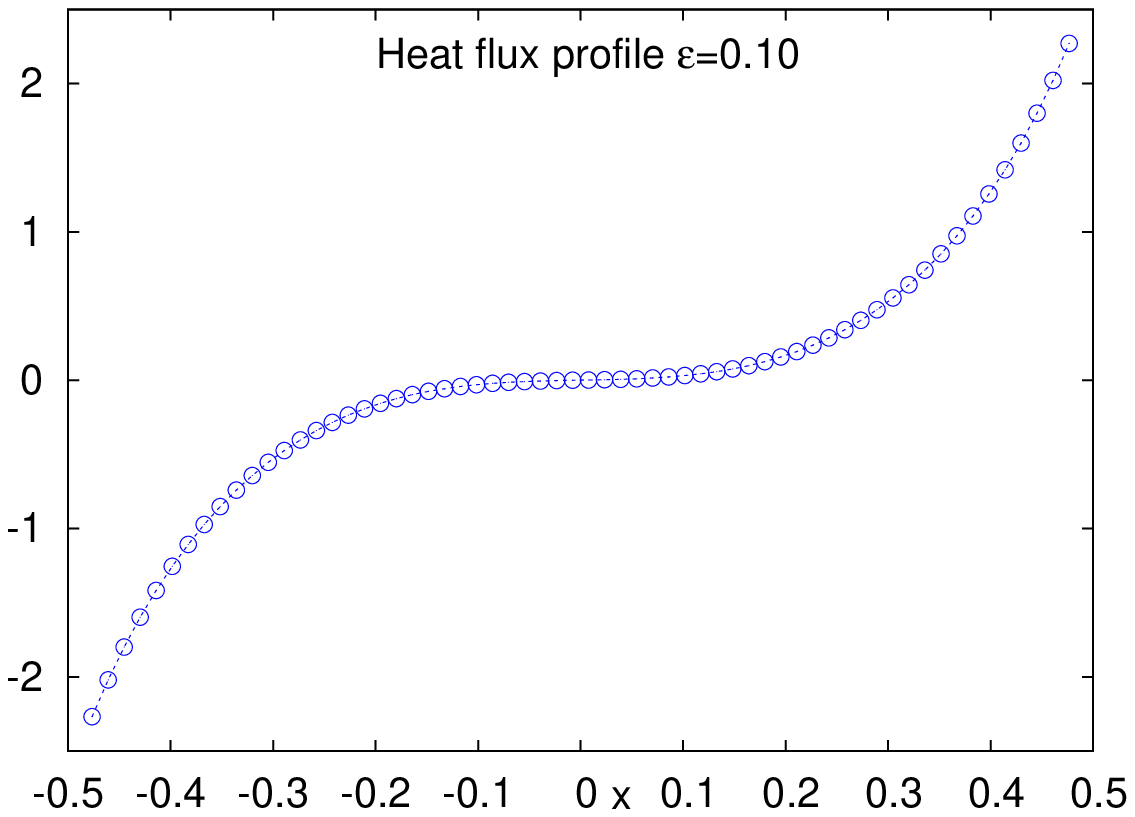}  
\\
(1) $\var=0.05$ &(2)  $\var=0.1$
\\
\includegraphics[width=7.5cm]{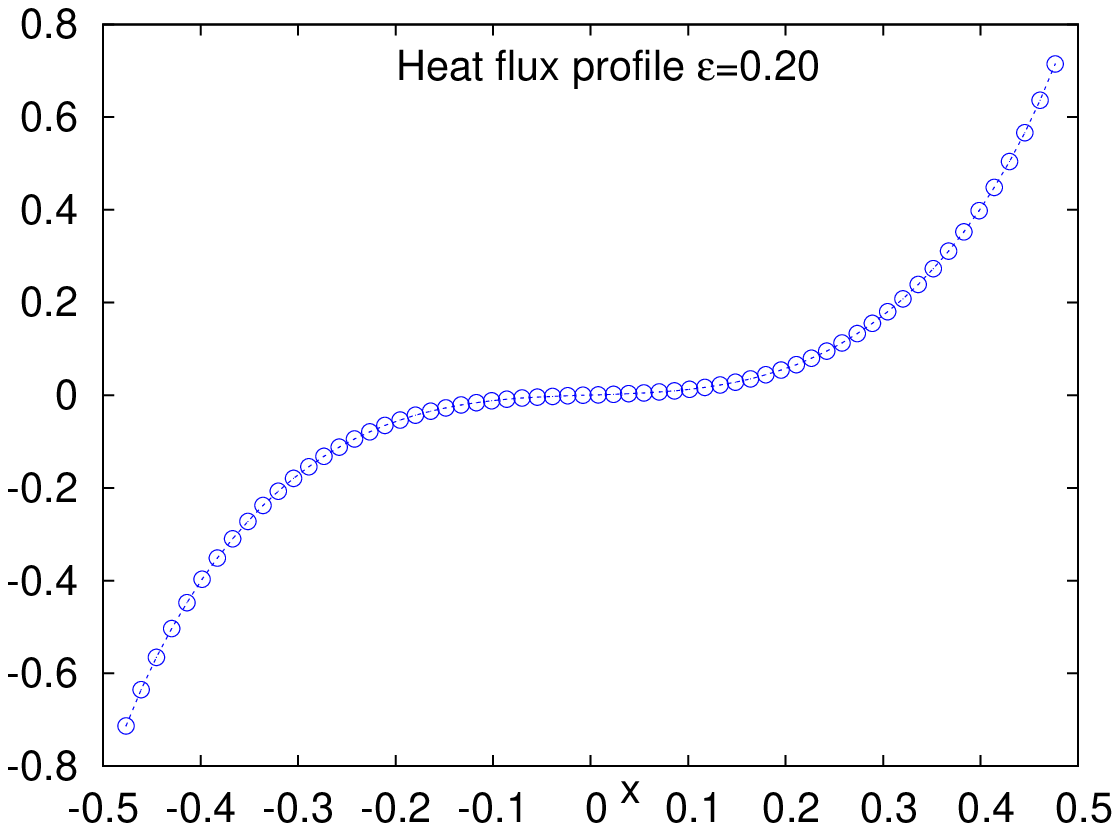} & 
\includegraphics[width=7.5cm]{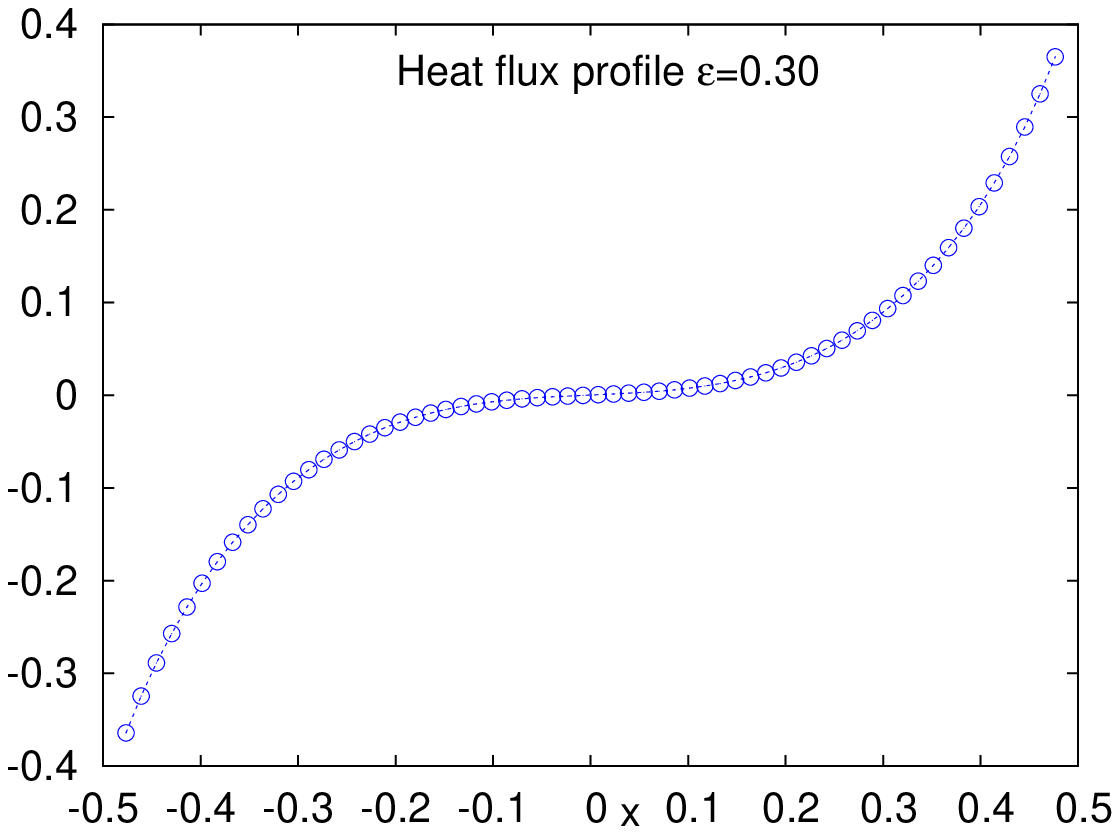}  
\\
(3) $\var=0.2$ &(4)  $\var=0.3$
\end{tabular}
\caption{Poiseuille flow: {\em heat flux for various Knudsen numbers $\var=0.05$, $0.1$, $0.2$ and $0.3$.}}
\label{Fig_pf_2}
\end{figure}

 The results are in good agreement with those presented in \cite{uribe} obtained from the BGK model for the Boltzmann equation. It clearly shows that Boltzmann's equation implies a gas-dynamics that has a more complex nature than standard hydrodynamics and most hydrodynamic quantities show boundary effects. Of course, in this simple configuration (a laminar stationary flow), it is still possible to derive analytic perturbative expressions for every hydrodynamic field and compare them with what is obtained in our simulations \cite{santos}. The effects beyond standard hydrodynamics is clearly observable, and are correctly described with our numerical scheme.

%%%%%%%%%%%%%%%%%%%%%%%%%%%%%%%%%%%%%%%%%%%%%%%%%%%%%%%%%%%%%%%%%%%%%%%%%%%%%%%%%%%%%%
%%%%                                                                             %%%%%
%%%%                test 4                                                       %%%%%
%%%%                                                                             %%%%%
%%%%%%%%%%%%%%%%%%%%%%%%%%%%%%%%%%%%%%%%%%%%%%%%%%%%%%%%%%%%%%%%%%%%%%%%%%%%%%%%%%%%%%

\subsection{The ghost effect}
\label{sec4-4}
It has been shown analytically \cite{Sone,Boby:95} and numerically \cite{SoneBob} on the basis of the kinetic theory that the heat-conduction equation is not suitable for describing the temperature field of a gas in the continuum limit  in an infinite domain without flow at infinity, where the flow vanishes in this limit. Indeed, as the Knudsen number of the system approaches zero, the temperature field obtained by the kinetic equation approaches that obtained by the asymptotic theory and not that of the heat-conduction equation, although the velocity of the gas vanishes. Here we review such a phenomenon, called the ghost effect, in the framework of the simulation of the Boltzmann equation.

We consider a gas  between two plates at rest in a finite domain. In this situation, the
stationary state at a uniform pressure (the velocity is equal to zero and the
pressure is constant) is an obvious solution of the Navier-Stokes equations; the
temperature field is determined by the heat conduction equation \cite{Sone}
$$
u = 0, \quad \rho\, T = C,\quad -\nabla_x\cdot(T^{1/2}\nabla_x\,T)=0.
$$
According to the Hilbert expansion with respect to the Knudsen number
$\var$, the density and temperature fields in the continuum limit are
affected by the velocity field, which is of order one with respect to
$\var$. Finally, the heat conduction equation, although extracted from
the incompressible Navier-Stokes system, is not appropriate in a whole class of
situations, in particular when isothermal surfaces are not parallel, thereby
giving rise to ``ghost effects''.

In this section, we will show that the numerical solution agrees with one
obtained by the asymptotic theory and not with the one obtained from the heat
conduction equation; this result is a confirmation of the validity of the
asymptotic theory. This problem has been already studied from the numerical
point of view for the time independent BGK operator, but not for the full time
dependent Boltzmann equation for hard sphere molecules.

Consider a rarefied gas between two parallel plane walls at $y=0$ and $y=1$.
Both walls have a common periodic temperature distribution $T_w$
$$
T_w(x) = 1 - 0.5\,\cos(2\,\pi \, x); \quad \forall x \in (0,1),
$$  
and a common small mean velocity $u_w$ of order $\var$ in its plane
$$
u_w(x) = (\var,0).
$$
On the basis of kinetic theory, we numerically investigate the behavior of
the gas, especially the temperature field, for various small Knudsen numbers
$\var$. Then, we will assume:
\begin{itemize}
\item the behavior of the gas is described by the Boltzmann equation for hard
  sphere molecules.
\item the gas molecules make diffuse reflection on the walls (complete accommodation).
\item the solution is 1-periodic with respect to $x$. Then, the average of
  pressure gradient in the $x$ direction is zero.
\end{itemize} 
In this example, the walls are moving with a speed of order $\var$. We apply the determinitic scheme in $2d_x\times 2d_v$ and choose a computational domain $[-7,7]\times [-7,7]$ in the velocity space with a number grid points  $n_v=32$ in each direction and for the space direction $n_x=n_y=50$. Finally we take $\Delta t=0.001$.

\begin{figure}[htbp]
\begin{tabular}{cc}
\includegraphics[width=7.5cm,height=7.cm]{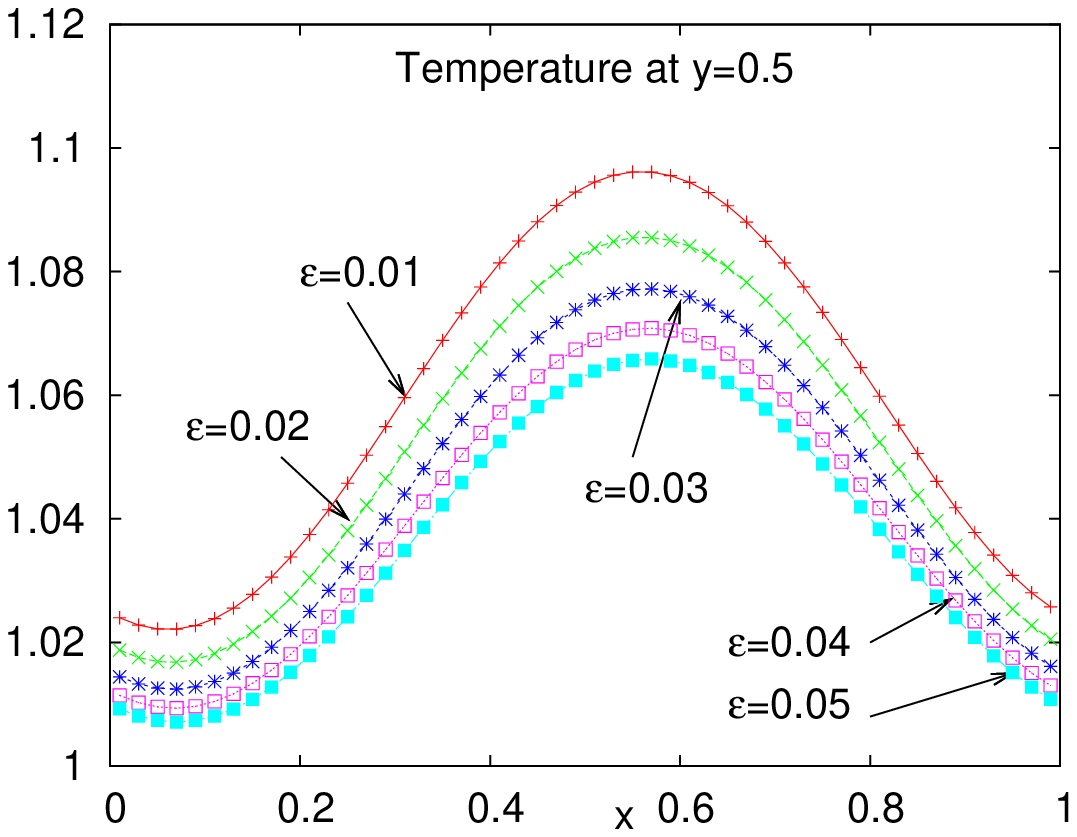}     
&
\includegraphics[width=7.5cm,height=7.cm]{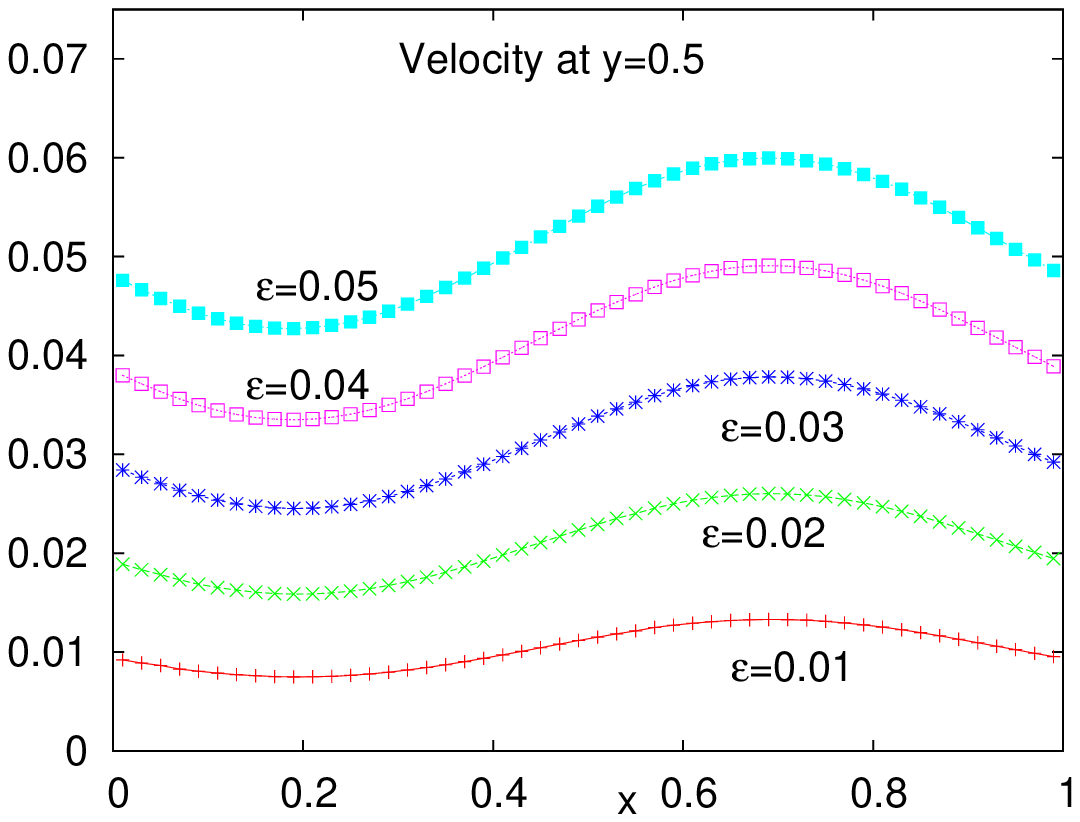}  
\\
(1)&(2)
\end{tabular}
\caption{Ghost effect: {\em temperature and mean velocity along $y=const$ 
    for various Knudsen numbers $\var=0.01$, $0.02$, $0.03$, $0.04$ and $0.05$, (1) temperature at $y=0.5$ (2) mean velocity $u$ at $y=0.5$.}}
\label{Fig_ge_1}
\end{figure}

\begin{figure}[htbp]
\begin{tabular}{cc}
\includegraphics[width=7.5cm,height=7cm]{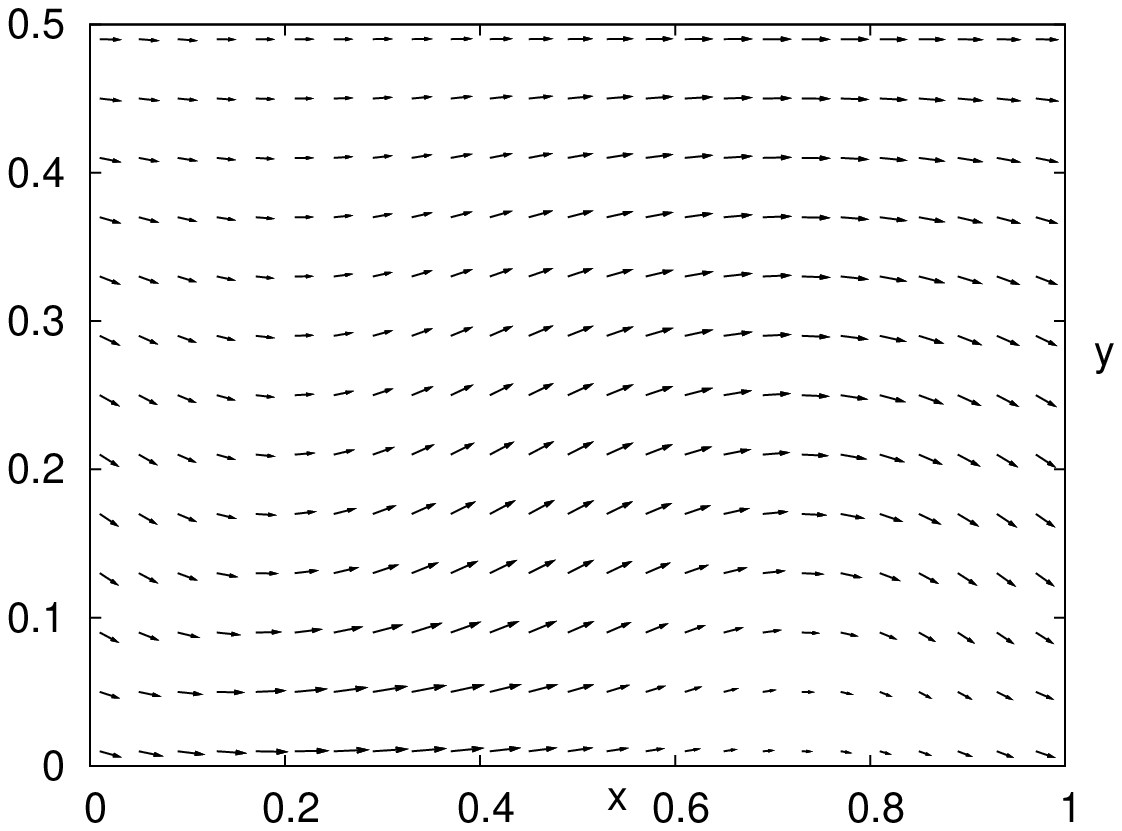}    
&
\includegraphics[width=7.5cm,height=7cm]{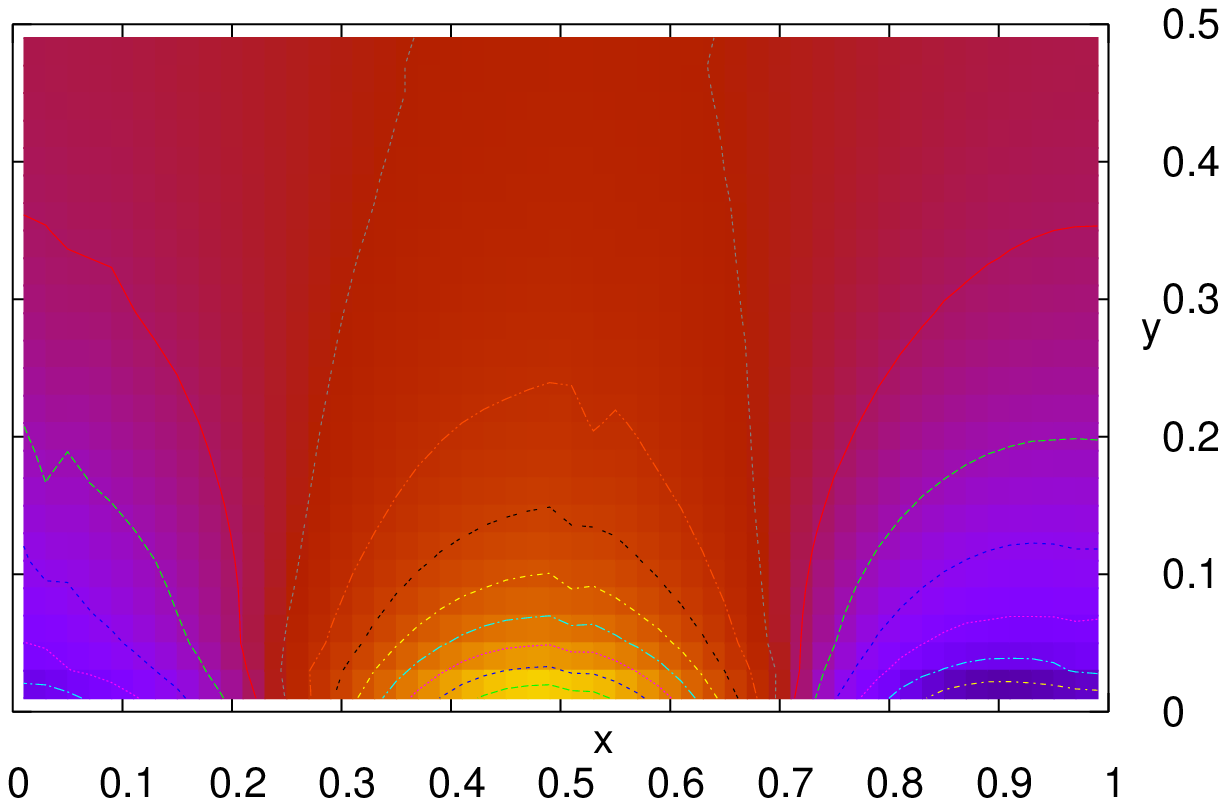}   
\\
(1)&(2)
\end{tabular}
\caption{Ghost effect: {\em (1) velocity field $u$, (2) isothermal lines for a fixed Knudsen number $\var=0.02$.}}
\label{fig7}
\end{figure}

The isothermal lines and the velocity field for $\var=0.02$ are shown in Figure
\ref{fig7}. These results are in good agreement with those obtained by
discretizing the the BGK operator \cite{Sone}. Moreover, according to the
numerical simulations presented in \cite{Sone}, the temperature field deviates
from one given by the heat conduction equation and is increasing when the
Knudsen number $\var$ goes to zero whereas the velocity flow is vanishing
(Figure \ref{Fig_ge_1}). 

The temperature converges to the temperature
given by the Asymptotic theory developed by Sone {\it et al.} \cite{Sone} and
not to the solution of the Heat equation. Let us note that in this particular
case, we cannot compute an accurate solution for very small Knudsen number, because the computational time to reach the stationary solution with a very good accuracy becomes too large.

\section{Conclusions}
\label{sec5}
In this paper we present an accurate deterministic method for the numerical
approximation of the space non homogeneous, time dependent Boltzmann equation in a bonded domain with different boundary conditions.

The method couples a second order finite volume scheme for the treatment of the transport step with a Fourier spectral method for the collision step.

It possesses a high order of accuracy for this kind of problems. In fact it is
second order accurate in space, and spectrally accurate in velocity. The high accuracy is evident from the quality of the numerical results that can be obtained with a relatively small number of grid points in velocity domain.

An effective time discretization allows the treatment of problems with a
considerable range of mean free path, and the decoupling between the transport
and the collision step makes it possible the use of parallel algorithms, which
become competitive with state-of-the-art numerical methods for the Boltzmann
equation.

The numerical results, and the comparison with other numerical results available in the literature, show the effectiveness of the present method for a wide class of problems.

%%-----------------------------
%%      your bibliography
%%-----------------------------

\begin{flushleft} 
\signff 
\end{flushleft}

\end{document}